\documentclass[11pt,letter]{scrartcl}
\usepackage{amsmath}
\usepackage{amssymb}
\usepackage{amsthm}
\usepackage{mathtools}
\usepackage{hyperref}
\usepackage[automark]{scrpage2}

\usepackage{tikz}
\usetikzlibrary{positioning}

 \newtheorem{thr}{Theorem}[section]
 \newtheorem{lem}{\bf Lemma}[section]
 \newtheorem{sleds}{\bf Corollary}[section]
 \newtheorem{prop}{\bf Proposition}[section]
 \newtheorem{defn}{\bf Definition}[section]
 
 \newtheorem{rem}{\bf Remark}[section]

 \usepackage{mathbbol}
\usepackage[mathscr]{euscript}
\newenvironment{posmallmatrix}
  {\left(\begin{smallmatrix}}
  {\end{smallmatrix}\right)}

\newcommand{\R}{\mathbb R}
\newcommand{\C}{\mathbb C}
\newcommand{\N}{\mathbb N}
\newcommand{\rb}{{\bf r}}
\newcommand{\tb}{{\bf t}}
\newcommand{\mb}{{\bf m}}
\newcommand{\lb}{{\bf l}}
\newcommand{\Mb}{{\bf M}}
\newcommand{\Lb}{{\bf L}}
\newcommand{\MM}{\pmb{\mathbb m}}
\newcommand{\LL}{\pmb{\mathbb l}}

\newcommand{\MMC}{\pmb{\mathbb M}}
\newcommand{\LLC}{\pmb{\mathbb L}}

\mathtoolsset{mathic=true}

\numberwithin{equation}{section}
\title{The Resolvent Matrix of the Truncated Hausdorff
 Matrix Moment Problem via New Dyukarev--Stieltjes
 Parameters and Extremal Solutions via Continued Fractions}
\author{
Abdon E. Choque-Rivero \footnote{A.~E.~Choque-Rivero is supported by
SNI--CONACYT and CIC--UMSNH, M\'exico.} }
\date{}
\begin{document}
\maketitle

\begin{abstract}
 We obtain a new multiplicative decomposition of the
 resolvent matrix of the truncated  Hausdorff matrix moment (THMM)
 problem in the case of an odd
 and even number of moments via new Dyukarev--Stieltjes matrix (DSM)
 parameters.
   Explicit interrelations between new DSM
 parameters
 and orthogonal matrix polynomials on a finite interval $[a,b]$, as
 well as the Schur complements of the block Hankel matrices
 constructed through the moments of the THMM problem are given.
 Additionally, the extremal solutions of the THMM problem are represented via matrix continued
 fractions in terms of the DSM parameters.
\end{abstract}

\subsection*{Keywords:}
Resolvent matrix, orthogonal matrix polynomials, Dyukarev-Stieltjes
parameters, matrix continued fractions.
%
%
\section{Introduction}
Throughout this paper, let $q$ and $p$ be positive integers. We will
use ${\mathbb C}$, ${\mathbb R}$, ${\mathbb N}_0$ and ${\mathbb N}$
to denote the set of all complex numbers, the set of all real
numbers, the set of all nonnegative integers, and the set of all
positive integers, respectively.  The notation ${\mathbb C}^{q\times
q}$ stands for the set of all complex $q\times q$ matrices. For the
null matrix that belongs to ${\mathbb C}^{p\times q}$ we
 will write $0_{p\times q}$.
 We denote by $0_{q}$ and $I_q$  the null and the identity   matrices
 in ${\mathbb C}^{q\times
 q}$, respectively.  In cases where the sizes of the null and the identity matrix are
 clear, we will omit the indices.

 \vskip2mm

The continued fraction
 \begin{align} \label{103}
\cfrac{1}{m_{1}\,z+\cfrac{1}{l_{1}+\cfrac{1}{m_{2}\,z+\cfrac{1}{l_{2}}+\cfrac{1}{\ddots
\, }}}
  }
 \end{align}
where 
$l_j$ and $m_j$ are positive numbers was used by T.J. Stieltjes in
\cite{stie} to determine whether there is a unique (resp.
non-unique) solution of the Stieltjes moment problem. This problem
is stated as follows: given a sequence $(s_k)_{k\geq0}$ of real
numbers, find the set ${\mathcal M}$ of positive measures $\sigma$
on $[0,\infty)$ such that $s_k =\int_0^\infty  x^k d\sigma(x)$ for
$k = 0,1,\ldots$. 
 Note that
 \begin{equation} \label{scalm}
 l_j:=\frac{\Delta_{j}^2}{\Delta_{j-1}^{(1)}
 \Delta_{j-2}^{(1)}},\quad \mbox{and} \quad
 m_j:=\frac{[\Delta_{j}^{(1)}]^2}{\Delta_j \Delta_{j-1}},
 \end{equation}
 with
  $\Delta_j:=\det (s_{i+k})_{i,k=0}^j$, \, $\Delta_j^{(1)}:=\det
  (s_{i+k+1})_{i,k=0}^j$ and ($\Delta_{-1}=\Delta_{-1}^{(1)}:=1$).
Here we use  Krein's notation \cite{kre1} for the nonnegative
coefficients $l_j$ and $m_j$, also called Stieltjes parameters.

It is well-known that the truncated continued fraction of
(\ref{103}) can be written as $F_p(z,0)$ where $F_p$ is a
composition of M\"{o}bius transformations: $F_p=f_1\circ f_2\circ
\ldots \circ f_p$ with $f_{2k-1}(z,\omega):=\frac{1}{-z
m_k+\omega}$, $f_{2k}(\omega):=\frac{1}{l_k+\omega}$ for each $k\in
\N$; see \cite[Page 474]{hen}.
 On the other hand, denoting the matrix of the transformations
 $f_{2k-1}$, $f_{2k}$
 by $t_{2k-1}:=\left(
           \begin{array}{cc} 0 & 1 \\ 1&-z m_{k}
           \end{array}         \right)$,
          $t_{2k}:=\left(
           \begin{array}{cc} 0 & 1 \\ 1&  l_{k}
           \end{array}
         \right)$, respectively, one can associate the M\"{o}bius
         transformation with the $2\times 2$ matrix
          \begin{equation}
         \mathscr{T}_{2n}:=
 t_1\,t_2\,\ldots\, t_{2n},
\label{0001}
          \end{equation}
or
 \begin{equation} \label{0002}
  \mathscr{T}_{2n+1}:=t_1\,t_2\,\ldots\,
          t_{2n+1}
          \end{equation}
for $\omega=0$. See \cite[Theorem 12.1a]{hen}.

In \cite{dyu0} Dyukarev introduced the matrix version of the
Stieltjes parameters in order to establish the determinateness of
the  Stieltjes matrix moment problem. An important feature used in
\cite{dyu0} is the similar factorization to (\ref{0001}) and
(\ref{0002}) of the resolvent matrix of the truncated Stieltjes
matrix moment problem via Stieltjes parameters.
 Recall that the mentioned resolvent matrix,
also called Nevanlinna matrix \cite[Definition 2]{dyu0}, is a
$2q\times 2q$ matrix polynomial.

\vskip2mm

 In the present work we introduce new matrix Stieltjes parameters,
 called Dyukarev-Stieltjes matrix (DSM) parameters of the truncated
  Hausdorff matrix moment (THMM) problem. With the help of the DSM parameters,
 we obtain a new multiplicative representation of the
resolvent matrix (RM):
 \begin{equation}
 U^{(m)}(z)=\left(\begin{array}{cc}
                 \alpha^{(m)}(z) & \beta^{(m)}(z) \\
                 \gamma^{(m)}(z) & \delta^{(m)}(z)
               \end{array}
 \right)
 \end{equation}
  of the THMM problem in the case of an odd and even number of moments.
The RM $U^{(m)}$ is a $2q\times 2q$  matrix polynomial which we
factorize as follows:
\begin{align}
U^{(2n)}=&\mathscr{D}_1 \LL_{-1}^{(2n)}\MM_0^{(2n)}\ldots
\MM_{n-1}^{(2n)} \LL_{n-1}^{(2n)}\mathscr{B}^{(2n)}_2 \mathscr{D}_2
 \label{0003}\\
U^{(2n+1)}=&\mathscr{D}_3 \LL_{-1}^{(2n+1)}\MM_0^{(2n+1)}\ldots
 \LL_{n-1}^{(2n+1)}\MM_{n}^{(2n+1)}\mathscr{B}^{(2n+1)}_2 {\bf D}_4
 \label{0004}
\end{align}
where $\mathscr{D}_k$  are anti-diagonal block matrices,
 ${\bf D}_4$ is a diagonal matrix,
$\mathscr{B}^{(2n)}_2$, $\mathscr{B}^{(2n+1)}_2$, $\LL_j^{(2n+1)}$,
$\MM_j^{(2n)}$ are constant anti-triangular  block matrices and
$\MM_j^{(2n+1)}$, $\LL_j^{(2n)}$ are affine on $z$ and
 anti-triangular block matrices;
 see Theorem \ref{mainT1} and Remark \ref{lem2016}. 
As a consequence of the representation of the RM as in (\ref{0003})
and (\ref{0004}), the so-called extremal solutions of the THMM
problem $\alpha^{(2n)}\gamma^{(2n)^{-1}}$,
$\beta^{(2n)}\delta^{(2n)^{-1}}$,
$\alpha^{(2n+1)}\gamma^{(2n+1)^{-1}}$ and
$\beta^{(2n+1)}\delta^{(2n+1)^{-1}}$ are given by a
 finite matrix continued fraction; see Theorem \ref{mainT2}.

\vskip2mm
 Our motivation is to give a complete characterization of solvability
 of the THMM problem via the DSM parameters.
 Such characterization will be
 considered elsewhere. Proposition \ref{propHnew} and \cite[Proposition 7]{abdonDS1}  are directed
 to the mentioned characterization.
 In the indicated propositions, the recovery of positive
 moment  sequences $(s_j)_{j=0}^m$ (as in Definition \ref{deboth})
  from the DSM parameters are given; see Definition \ref{st2ndef-2np1}.

\vskip3mm

 Let us now summarize the notions appearing in the last two
paragraphs.

\vskip2mm

The THMM problem is stated as follows: given an interval $[a,b]$ on
the real axis and a finite sequence of $q\times q$ matrices,
$(s_j)_{j=0}^m$. Describe the set ${\mathcal
M}_\geq^q[[a,b],\mathfrak B\cap[a,b];
  (s_j)_{j=0}^m]$ of all nonnegative Hermitian $q\times q$ measures
  $\sigma$ defined on the $\sigma$-algebra of all Borel subsets of
  the interval $[a,b]$ such that
$$
s_j=\int\limits_{[a,b]} t^j d\sigma(t)
 $$
 holds true for each integer $j$ with $0\leq j\leq m$.

Let $(s_j)_{j=0}^{2n}$ (resp. $(s_j)_{j=0}^{2n+1}$) be a sequence of
complex $q\times q$ matrices, then denote
\begin{align}
 H_{1,j}:=\widetilde{H}_{0,j},\, j\geq0,\ \
H_{2,j-1}:=-ab\widetilde{H}_{0,j-1}+(a+b)\widetilde{H}_{1,j-1}
-\widetilde{H}_{2,j-1}, \, j\geq 1 \label{63}
\end{align}
and
\begin{equation}\label{2nmas1}
  K_{1,j}=b\widetilde{H}_{0,j}-\widetilde{H}_{1,j}
  , \quad
  K_{2,j}=-a\widetilde{H}_{0,j}+\widetilde{H}_{1,j}, \quad  j\geq 0,
 \end{equation}
where $\widetilde H_{\cdot,j}$
 are defined with the help of the Hankel matrices,
$$
\widetilde{H}_{0,j}:=\{s_{l+k}\}_{l,k=0}^{j},
 \quad \widetilde{H}_{1,j}:=\{s_{l+k+1}\}_{l,k=0}^{j},
 \quad \mbox{and} \quad
 \widetilde{H}_{2,j}:=\{s_{l+k+2}\}_{l,k=0}^{j}.
$$
In \cite[Theorem 1.3]{abdon2} (resp. \cite[Theorem 1.3]{abdon1}), it
was demonstrated that there exists a solution to
 problem ${\mathcal M}_\geq^q[[a,b],\mathfrak B\cap[a,b]; (s_j)_{j=0}^{2n}]$
 (resp. ${\mathcal M}_\geq^q[[a,b],\mathfrak B\cap[a,b]; (s_j)_{j=0}^{2n+1}]$
 if and only if the block matrices $H_{1,n}$ and
$H_{2,n-1}$ (resp. $K_{1,n}$ and $K_{2,n}$) are  both nonnegative
Hermitian, where the problem of finding the set ${\mathcal
M}_\geq^q[[a,b],\mathfrak B\cap[a,b]; (s_j)_{j=0}^{m}]$ for $m=2n$
and $m=2n+1$ is usually reduced to 
 searching for  the set of holomorphic functions
 \begin{align*}
{\mathfrak S}_\geq^q[[a,b],\mathfrak B\cap[a,b];
 (s_j)_{j=0}^{m}] 
  :=\left\{
 s(z)=\int_{[a,b]}\frac{d\sigma(t)}{t-z},\,
 \sigma\in {\mathcal M}_\geq^q[[a,b],\mathfrak B\cap[a,b];
 (s_j)_{j=0}^{m}]
 \right\}.
 \end{align*}
 If $H_{1,n}$ and $H_{2,n-1}$ (resp. $K_{1,n}$ and $K_{2,n}$) are
 positive Hermitian, the set
 ${\mathfrak S}_\geq^q[[a,b],\mathfrak
 B\cap[a,b];
 (s_j)_{j=0}^{m}]$ is  parameterized via a linear fractional
 transformation
\begin{equation}
 s(z)=(\alpha^{(m)}(z){\bf p}(z)+\beta^{(m)}(z){\bf q}(z))
 (\gamma^{(m)}(z){\bf p}(z)+\delta^{(m)}(z){\bf q}(z))^{-1}. %
 \label{0002a}
\end{equation}
The column pair $({\bf p}, {\bf q})$ satisfies certain properties in
every case; see Definitions \cite[Definition 5.2]{abdon1}
 and \cite[Definition 5.2]{abdon2}.

 The $2q\times 2q$ polynomial matrices
\begin{align}
U^{(2j)}(z,a,b):=&\left(
 \begin{array}{cc}
   \Theta_{2,j}^*(\bar z,a)\Theta_{2,j}^{*^{-1}}(a,a) & \frac{1}{b-a}
   \Theta_{1,j}^*(\bar z,b)
   \Gamma_{1,j}^{*^{-1}}(a,b) \\
   (z-a)\Gamma_{2,j}^*(\bar z,a)\Theta_{2,j}^{*^{-1}}(a,a)
    &\frac{b-z}{b-a}
    \Gamma_{1,j}^*(\bar z,b)\Gamma_{1,j}^{*^{-1}}(a,b) \\
 \end{array}
 \right)\label{uujd}
 \end{align}
 and
 \begin{align}
U^{(2j+1)}(z,a,b)
:= \left(
 \begin{array}{cc}
   Q_{2,j}^*(\bar z,a,b)Q_{2,j}^{*^{-1}}(a,b,a)
   & -Q_{1,j+1}^*(\bar z)P_{1,j+1}^{*^{-1}}(a) \\
   -(z-a)(b-z)P_{2,j}^*(\bar z,a,b)Q_{2,j}^{*^{-1}}(a,b,a)
   &P_{1,j+1}^*(\bar z)P_{1,j+1}^{*^{-1}}(a) \\
 \end{array}
 \right)\label{uujdeven}
\end{align}
 are called the resolvent matrix (RM) of the THMM problem.
  The $q\times q$ matrix polynomials  $P_{k,j}$, $Q_{k,j}$, $\Gamma_{k,j}$
  and $\Theta_{k,j}$ for $k=\{1,2\}$
  are constructed
 via the given data: the sequence of moments $(s_j)_{j=0}^{2n}$ (resp.
 $(s_j)_{j=0}^{2n+1}$). See Definition \ref{de002} and \ref{def001}.

   In \cite{abdonDS1}, the relationships
 $U^{(2j+1)}= \left(
  \begin{array}{cc}
    \frac{1}{z-a}I_q & 0_q \\
    0_q & I_q
  \end{array}
\right) \widetilde U_1^{(2j+1)}A^{(2j+1)}$ $\cdot\left(
                                           \begin{array}{cc}
                                             (z-a)I_q & 0_q \\
                                             0_q & I_q
                                           \end{array}
                                         \right)
$ and $U^{(2j)}=\widetilde U_1^{(2j)}A^{(2j)}$
 were used, with $A^{(k)}$ denoting $2q\times 2q$ matrices
depending on $a$ and $b$ in order to factorize the matrices
$U^{(m)}$ for $m$ odd and even.
 Instead of the mentioned relations, in the
present paper we employ the following two relations:
\begin{align}
 U^{(2j)}(z)=
 \left(
               \begin{array}{cc}
                 \frac{1}{(z-a)(b-z)}I_q & 0_q \\
                 0_q & I_q \\
               \end{array}
             \right)\left(
               \begin{array}{cc}
                 I_q & s_0 z \\
                 0_q & I_q \\
               \end{array}
             \right)
             \widetilde U_2^{(2j-2)}(z)A_2^{(2j)}\left(
                                   \begin{array}{cc}
                                     (b-a)(z-a)I_q & 0_q \\
                                     0_q & \frac{b-z}{b-a}I_q \\
                                   \end{array}
                                 \right).\label{eq29aa}
\end{align}
and
\begin{align}
 U^{(2j+1)}(z)
 =\left( \begin{array}{cc}
 \frac{1}{b-z}I_q & 0_q \\  0_q & I_q \\ \end{array}
             \right) \widetilde U_2^{(2j+1)}(z)A_2^{(2j+1)}\left(
                                   \begin{array}{cc}
                                     (b-z)I_q & 0_q \\
                                     0_q & I_q
                                   \end{array}
                                 \right),\label{eq29aa0}
\end{align}
 where $\widetilde U_2^{(k)}(z)$, $A_2^{(k)}(z)$ for $k=2j$
 ($k=2j+1$)
 are introduced in (\ref{leipRMe22}), (\ref{RM100}),
 (\ref{leipB12nm1}) and (\ref{lei6.16}).
 Equalities (\ref{eq29aa}) and (\ref{eq29aa0}) are the consequence of
  \cite[Equality (6.26)]{abdon2} and
   \cite[Equalities (6.26),(6.27)]{abdon1}.

Note that the auxiliary matrices $\widetilde U_1^{(2j)}$ and
$\widetilde U_1^{(2j+1)}$ (resp. $\widetilde U_2^{(2j)}$ and
$\widetilde U_2^{(2j+1)}$) are related to $H_{1,j}$ and $K_{2,j}$
(resp. $H_{2,j}$ and $K_{1,j}$),
correspondingly. 

  As in \cite[Corollary 1 and Corollary 2]{abdonDS1} where the auxiliary
   matrix $\widetilde U_1^{(k)}$ was decomposed,
  in the present work  we factorize
 the auxiliary matrix $\widetilde U_2^{(2n+1)}$ in the following
  form (as in Corollary
  \ref{cor-2np1}):
 \begin{align}
 \widetilde
U_2^{(2n+1)}=&d^{(1)}d^{(3)}\ldots d^{(2n-1)}d^{(2n+1)}
\label{uuu2nm1}
 \end{align}
Instead of  $\widetilde U_2^{(2n)}$ the auxiliary
  matrix $\widehat U_2^{(2n)}$ (as in (\ref{RM1000})) is used. The factorization
 \begin{align}
 \widehat U_2^{(2n)}=d^{(0)}d^{(2)}\ldots d^{(2n-4)}d^{(2n-2)}
\label{uuu2n}
 \end{align}
 is employed to prove a new factorization of the RM $U^{(2n)}$.
 The matrices $d^{(2k+1)}$ and $d^{(2k)}$ are  affine on $z$.
 The importance of the  auxiliary matrices $\widetilde U_2^{(2j+1)}$ and
 $\widehat U_2^{(2j)}$ resides in the fact that they belong to the Potapov
  class of  matrix functions \cite{pot0}. The matrix valued functions
  belonging to this class  can be factorized into elementary
 factors, as seen in Corollary \ref{utildes}.
  A similar factorization was revealed by Yu. Dyukarev in \cite{dyu001} by
 developing a multiplicative representation of the RM of the
 truncated Stieltjes matrix moment (TSMM) problem.
Generalized Stieltjes parameters for Nevanlinna-Pick  interpolation
problems in certain Nevanlinna classes  were considered in
 \cite{dyu11} and \cite{seri2}.
 The determinateness of the TSMM problem was
 obtained in \cite{dyu0} with the help of Dyukarev--Stieltjes matrix
 parameters of the TSMM.  In
\cite{dyu2015} a criterion for complete indeterminacy of limiting
Stieltjes interpolation problem in terms of orthonormal matrix
functions was achieved.
 In \cite{abdyu1}, by using a decomposition of the RM of the TSMM problem,
 the following were demonstrated:
 necessary and sufficient conditions for the TSMM
  problem to have a unique solution and infinitely many
solutions for the Hamburger moment problem with the same moments.
 Note that in \cite{zag} and \cite{abZag} the operator approach was
employed to solve  the THMM.

\vskip2mm

 The main results of the present work
  are the new factorization of the RM $U^{(m)}$ presented in
  Theorem \ref{mainT1} and the
  representation of the extremal solutions with the help
  of matrix continued fractions in terms of
  DSM parameters; see Theorem \ref{mainT2}.
Matrix continued fractions were studied in \cite{ApNi},
\cite{sorokin}, \cite{zha}, \cite{zyg}, \cite{simon} and references
therein.

\vskip1mm

   Additionally, the following facts are
  attained:
\begin{itemize}
\item[(i)] We obtain a multiplicative representation of the auxiliary matrices
$\widetilde U_2^{(m)}$ (as in Corollary \ref{cor-2np1}) via new
auxiliary Blaschke--Potapov factors $d^{(m)}$;  see (\ref{leipg0}),
(\ref{leipgj}), (\ref{leipfj}).
\item[(ii)] Each Blaschke--Potapov factor $d^{(2j)}$ (resp. $d^{(2j+1)}$)
  is decomposed via new Dyukarev--Stieltjes parameters type matrices
 $\mb_k$ and $\rb_k$ (resp. $\tb_k$ and $\lb_k$).
  See  Theorem \ref{mth001-2np1new}.
 \item[(iii)]
   Explicit relations between $P_{2,j}$, $Q_{2,j}$, $\Gamma_{1,j}$,
   $\Theta_{1,j}$ at $z=a$ and DSM parameters $\mb_j$, $\lb_j$ are given. See
  Proposition \ref{propa1new} and Remark \ref{remaa1new}.
 \item[(iv)]
  Explicit relations between the Schur complements $\widehat
  H_{2,j}$, $\widehat K_{1,j}$ (as in (\ref{32aa}), (\ref{32aa-2np1}))
   and the DSM parameters $\mb_j$, $\lb_j$
  (as in (\ref{lkk2}), (\ref{lkk3}))
  are given. On the other hand,
  explicit relations between the Schur complement
$\widehat H_{2,j}$, $\widehat K_{1,j}$ and polynomials $P_{2,j}$,
$Q_{2,j}$, $\Gamma_{1,j}$, $\Theta_{1,j}$ at $z=a$ are obtained. See
Remark \ref{rem4.4Anew}, Remark \ref{rem4.4BBnew}, and Proposition
\ref{propHbbnew}.
\item[(v)]
 Given $s_0$ and a sequence of  DSM parameters $\mb_j$, $\lb_j$,
 we obtain
 a sequence $(s_l)_{l=0}^{2j+1}$ (resp. $(-ab s_l+(a+b)s_{l+1}-s_{l+2})_{l=0}^{2j-2}$)
 such that $K_{1,j}$  (resp. $H_{2,j-1}$)
  is a positive definite matrix. See Proposition \ref{propHnew}.
\end{itemize}
To the author's knowledge,  the results
 contained in (i) through (v) are also new in the scalar case. In
 Section \ref{secscalar},
 the scalar version of the DSM parameters $\mb_j$ and $\lb_j$
 via determinants  is indicated.

  In comparison to the DSM parameters $M_k$ and $L_k$ \cite{abdonDS1}, the
  new DSM parameters $\mb_j$
  and  $\lb_j$ depend on both terminal points of the interval $[a,b]$.
 Other DSM parameters
 which also depend on $a$ and $b$ were introduced in \cite{abEVEN}.
 In turn the mentioned parameters are different from the ones studied in
 \cite{abdonDS1} (also in \cite{abODD}), where the parameters depend
 only on $a$.
 In Remark \ref{remlimitab} by setting $b\to+\infty$ and  $a=0$
 in the DSM parameters $\mb_j$ and $\lb_j$,
 we obtain the Dyukarev--Stieltjes parameters
 of the TSMM  problem \cite{dyu0}.

Throughout  the paper we decisively use the forms (\ref{uujd}) and
(\ref{uujdeven}) of the RM of the THMM problem obtained in
\cite{abKN} where the elements of the RM are given with the help of
four orthogonal polynomials and their second kind polynomials.
 Orthogonal matrix polynomials (OMP) were first considered by M.G. Krein in~1949~\cite{k1949},
\cite{k1949a}. Further investigations of OMP were made by
 J.S. Geronimo \cite{gero}, I.V.
Kovalishina~\cite{kov0}, \cite{kov1},  H. Dym \cite{dym1}, B. Simon
\cite{simon1}, Damanik/\-Pushnitski/\-Simon~\cite{simon}  and the
references therein. See also \cite{duranker}, \cite{duran},
\cite{duran1}, \cite{duran2}, \cite{duran3}, \cite{fkm},
\cite{dette2}, \cite{miran}, \cite{abdyu1} and \cite{abmad}.

 A brief outline of the relationships between the DSM parameters
  $\mb_j$, $\lb_j$
 and the polynomials $P_{2,j}, Q_{2,j}, \Gamma_{1,j}, \Theta_{1,j}$
 as well as the matrices $H_{2,j-1}, K_{1,j}, \widehat H_{1,j-1}, \widehat
 K_{1,j}$ achieved in the present work is given by

  
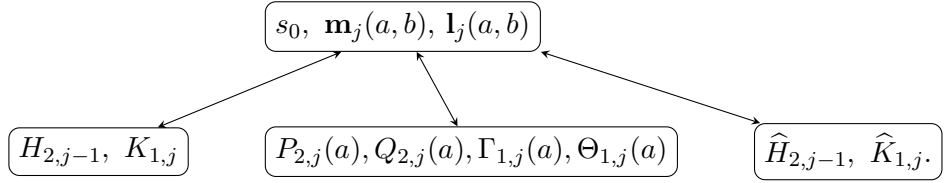
\begin{center}
\begin{tikzpicture}[>=stealth,every node/.style
={shape=rectangle,draw,rounded corners},]
    \node (c4) {$s_0$,\, $\mb_j(a,b)$, $\lb_j(a,b)$};
    \node (c5) [below left=of c4]{$H_{2,j-1}$,\, $K_{1,j}$};
    \node (c6) [right =of c5]{$P_{2,j}(a), Q_{2,j}(a), \Gamma_{1,j}(a),
     \Theta_{1,j}(a)$};
    \node (c7) [right =of c6]{$\widehat H_{2,j-1}$,\, $\widehat K_{1,j}$.};
    \draw[<->] (c4) -- (c5);
    \draw[<->] (c4) -- (c6);
    \draw[<->] (c4.south east) -- (c7);
\end{tikzpicture}
\end{center}

Here $A \longleftrightarrow B$ means both sides posses an explicit
interrelation between $A$ and $B$.

 The mentioned relationships complete in some sense the ones
obtained in \cite{abdonDS1}:

\begin{center}
\begin{tikzpicture}[>=stealth,every node/.style
={shape=rectangle,draw,rounded corners},]
    \node (c4) {$M_j(a)$, $L_j(a)$};
    \node (c5) [below left=of c4]{$H_{1,j}$,\, $K_{2,j}$};
    \node (c6) [right =of c5]{$P_{1,j}(a), Q_{1,j}(a), \Gamma_{2,j}(a),
     \Theta_{2,j}(a)$};
    \node (c7) [right =of c6]{$\widehat H_{1,j}$,\, $\widehat K_{2,j}$.};
    \draw[<->] (c4) -- (c5);
    \draw[<->] (c4) -- (c6);
    \draw[<->] (c4.south east) -- (c7);
\end{tikzpicture}
 \end{center}

\vskip1mm

 An application of
 the Hausdorff moment problem method, in particular of the solution
 set (\ref{0002a}) with (\ref{uujd}), (\ref{uujdeven})
 in the scalar case is used in \cite{ieeeA1} to solve the  admissible bounded control
 problem of the Brunovsky  control system of dimension $m$ for
 $m\geq2$ via orthogonal polynomials on  $[0,T]$ and their second
 kind polynomials. See also \cite{optimo} and \cite{cks2010}.
\section{Notations and preliminaries}\label{sec002}
 In this section we include the main notations and objects which we
 use throughout the paper. The auxiliary RM  $\widehat U_2^{(2j)}$ is
 introduced which is used instead of the auxiliary RM
 $\widetilde U_2^{(2j)}$, previously defined in
 \cite[Formula (6.2)]{abdon2}.

 The orthogonal matrix polynomials $P_{k,j}$, $\Gamma_{k,j}$ on $[a,b]$
   as well as their second kind polynomials $Q_{k,j}$, $\Theta_{k,j}$
  are recalled. The mentioned matrix polynomials together with the
  connection between the auxiliary RM $\widetilde U_2^{(2j+1)}$,
  $\widehat U_2^{(2j)}$ and the RM $U^{(m)}$ play an
  important role in the present work.
 \subsection{Auxiliary matrices and auxiliary resolvent matrices}
 Let $R_{j}:\mathbb{C}\to
\mathbb{C}^{(j+1)q\times (j+1)q}$
 be given by
\begin{align}
\label{52}R_{j}(z)&:=(I_{(j+1)q}-zT_{j})^{-1},\quad j\geq 0,
\end{align}
with
$ T_{0}:=0_{q}, \ \
 T_{j}:=\left( \begin{array}{cc}
0_{q\times jq} & 0_{q}\\
I_{jq} & 0_{jq\times q}\\
                \end{array}
         \right),\ \ j\geq 1.
 $

 Let
\begin{equation}
 v_{0}:=I_q,\quad
  v_{j}:=\left( \begin{array}{c}
I_q\\
0_{jq\times q}\\
\end{array}\right)=\left( \begin{array}{c}
v_{j-1}\\
0_{q}\\
\end{array}\right),
\quad\forall j\geq0. \label{59}
\end{equation}
Furthermore,  let
\begin{eqnarray}
\label{64} y_{[j,k]}&:=&\left(\begin{array}{c}
s_{j}\\
s_{j+1}\\
\vdots\\
s_{k}
\end{array}\right), \, 0\leq j\leq k\leq 2n. 
\end{eqnarray}

 Let
\begin{gather}
 \widetilde u_{1,0}:=s_{0},\, \quad \widetilde
u_{2,0}:=-s_{0}, \nonumber\\
\widetilde u_{1,j}:=y_{[0,j]}-b \left(\begin{array}{c}
                                              0_{q}\\
                                              y_{[0,j-1]}
                                             \end{array}\right),
 \quad
\widetilde u_{2,j}:=-y_{[0,j]}+a \left(\begin{array}{c}
                                              0_{q}\\
                                              y_{[0,j-1]}
                                             \end{array}\right)
  \label{22}
\end{gather}
for every $1\leq j\leq n-1$. In addition, for $1\leq j\leq n$ let
\begin{equation}\label{27}
 \widetilde Y_{1,j}:=b y_{[j,2j-1]}- y_{[j+1,2j]},
\quad \widetilde Y_{2,j}:= -a\,y_{[j,2j-1]}+y_{[j+1,2j]}.
\end{equation}
Let $\widehat K_{1,j}$  (resp. \,$\widehat K_{2,j}$) denote the
Schur complement of the block $bs_{2j}-s_{2j+1}$
(resp.\,$-as_{2j}+s_{2j+1}$) of the matrix $K_{1,j}$ (resp.\,
$K_{2,j}$). In addition, denote
\begin{align}
\widehat K_{1,0}=bs_{0}-s_{1},\, \widehat
K_{1,j}:=&bs_{2j}-s_{2j+1}-\widetilde
Y^{*}_{1,j}K^{-1}_{1,j-1}\widetilde Y_{1,j},\,
        \, 1\leq j\leq n,\label{lkk3}\\
\widehat K_{2,0}=-as_{0}+s_{1}, \,\widehat
K_{2,j}:=&-as_{2j}+s_{2j+1}-\widetilde
Y^{*}_{2,j}K^{-1}_{2,j-1}\widetilde Y_{2,j},\,
        \, 1\leq j \leq n. \label{lkk4}
 \end{align}
The quantities (\ref{lkk3}) and (\ref{lkk4}) have been defined in
\cite{dette2} for $a = 0$ and $b = 1$.

Let
$$
u_{1,0}:=0_{q},\, u_{1,j}:=\left(\begin{array}{c}
0_{q}\\
-y_{[0,j-1]}\\
\end{array}\right),\quad 1\leq j\leq n
$$
and
 \begin{equation}
 u_{2,0}:=-(a+b)s_0+s_1, \quad
u_{2,j}:=\left(\begin{array}{c}
u_{2,0}\\
-\widehat y_{[0,j-2]}\\
\end{array}\right),\, 1\leq j\leq 2n. \label{uuu001}
 \end{equation}
Moreover, let
\begin{equation}
\widehat{s}_{j}:=-abs_{j}+(a+b)s_{j+1}-s_{j+2},\quad 0\leq j\leq
2n-2\label{83}
\end{equation}
and
$$
 \widehat y_{[j,k]}:=\left(\begin{array}{c}
\widehat s_{j}\\
\widehat s_{j+1}\\
\vdots\\
\widehat s_{k}
\end{array}\right), \, 0\leq j\leq k\leq 2n-2.
$$
 Note that by (\ref{64}) and (\ref{83})
$$
 \widehat y_{[j,k]}=-aby_{[j,k]}+(a+b)y_{[j+1,k+1]}-y_{[j+2,k+2]}.
 $$
We also denote
$$
  Y_{1,j}:=y_{[j,2j-1]},\, 1\leq j\leq n, \,
                            \quad Y_{2,j}:=\widehat y_{[j,2j-1]},\, 1\leq j\leq
                            n-1.
$$
 Let $\widehat H_{1,j}$  (resp. \,$\widehat
H_{2,j}$) denote the Schur complement of the block $s_{2j}$
(resp.\,$\widehat s_{2j-2}$) of the matrix $H_{1,j}$ (resp.
$H_{2,j}$):  denote $\widehat H_{1,0}=s_{0}, \ \ \widehat
H_{2,0}=\widehat s_{0}$ and
\begin{align} \widehat
H_{1,j}:=&s_{2j}-Y^{*}_{1,j}H^{-1}_{1,j-1}Y_{1,j},\,
        \, 1\leq j\leq n,\label{lkk}\\
\widehat H_{2,j}:=&\widehat
s_{2j}-Y^{*}_{2,j}H^{-1}_{2,j-1}Y_{2,j},\,
        \, 1\leq j \leq n-1. \label{lkk2}
 \end{align}
The quantities (\ref{lkk}) and (\ref{lkk2}) have been defined in
\cite{dette2} for $a = 0$ and $b = 1$.

\begin{defn} \label{deboth} Let $[a,b]\subset \R$.
 The sequence $(s_k)_{k=0}^{2j}$ (resp. $(s_k)_{k=0}^{2j+1}$)
  is called a Hausdorff positive definite
   sequence  if the block Hankel matrices $H_{1,j}$ and $H_{2,j-1}$
 (resp. $K_{1,j}$ and $K_{2,j}$)  are both positive
 definite matrices.
\end{defn}
In the sequel, we will consider only  Hausdorff positive definite
sequences. In this case the THMM problem is called a non degenerate
THMM problem.

\begin{defn}\cite[Formula (6.2)]{abdon1} \label{rmeven22}
 Let $(s_k)_{k=0}^{2j+1}$ be a Hausdorff positive
  sequence. The $2q\times 2q$ matrix polynomial
\begin{equation}
\widetilde U_2^{(2j+1)}(z,a,b):=\left(\begin{array}{cc}
        \widetilde\alpha_2^{(2j+1)}(z) & \widetilde\beta_2^{(2j+1)}(z)\\
        \widetilde \gamma_2^{(2j+1)} (z)& \widetilde \delta_2^{(2j+1)}(z)
        \end{array}\right),\, z\in {\mathbb C}, \quad 1\leq j\leq n,
        \label{leipRMe22}
            \end{equation}
            with
\begin{align*}
 \widetilde \alpha_2^{(2j+1)}(z,a,b):=& I_q-(z-a)\widetilde u^{*}_{1,j}R_{j}^*(\bar{z})K^{-1}_{1,j}R_{j}(a)v_{j},
 \\
\widetilde \beta_2^{(2j+1)}(z,a,b):=&(z-a)\widetilde
u^{*}_{1,j}R_{j}^*(\bar{z})K^{-1}_{1,j}R_{j}(a)\widetilde u_{1,j},
\\
\widetilde \gamma_2^{(2j+1)}(z,a,b):=
&-(z-a)v^{*}_{j}R_{j}^*(\bar{z})K^{-1}_{1,j}R_{j}(a)v_{j},
  \end{align*}
  and
  \begin{align*}
\widetilde \delta_2^{(2j+1)}(z,a,b):=&
 I_q+ (z-a)v^{*}_{j}R_{j}^*(\bar{z})K^{-1}_{1,j}R_{j}(a)
 \widetilde u_{1,j}
 \end{align*}
 is called the second auxiliary matrix of the THMM problem in the case of an even number
 of moments.
\end{defn}
Let
      \begin{align*}
       B_{2,j}:=
       (b-a)\widetilde u_{2,j}^*R_j^*(a)K_{2,j}^{-1}
       R_j(a)\widetilde u_{2,j} 
       \end{align*}
       and
       \begin{align}
       A_2^{(2j+1)}:=&\left(
                    \begin{array}{cc}
                      I_q & B_{2,j} \\
                      0_q & I_q \\
                    \end{array}
                  \right). \label{leipB12nm1}
 \end{align}
 In \cite{abdon1}, the following equality was proved:
 $$
  U^{(2j+1)}=\left(
               \begin{array}{cc}
                 \frac{1}{b-z}I_q & 0_q \\
                 0_q & I_q \\
               \end{array}
             \right)\widetilde U_2^{(2j+1)}A_2^{(2j+1)}
             \left(
               \begin{array}{cc}
                 (b-z)I_q & 0_q \\
                 0_q & I_q \\
               \end{array}
             \right). 
$$
\begin{defn}\cite[Formula (6.2)]{abdon2} Let $(s_k)_{k=0}^{2j}$
 be a Hausdorff positive sequence. The $2q\times 2q$ matrix polynomial
\begin{equation}
\widetilde U_2^{(2j)}(z,a,b):=\left(\begin{array}{cc}
        \widetilde\alpha_2^{(2j)}(z,a,b) &
         \widetilde\beta_2^{(2j)}(z,a,b)\\
        \widetilde\gamma_2^{(2j)} (z,a,b)&
        \widetilde\delta_2^{(2j)}(z,a,b)
        \end{array}\right),\, z\in {\mathbb C}, \quad 0\leq j\leq n-1,
        \label{RM100}
            \end{equation}
            with
\begin{align*}
 \widetilde\alpha_2^{(2j)}(z,a,b):=& I_q-(z-a)u^{*}_{2,j}R_{j}^*(\bar z) H^{-1}_{2,j}R_{j}(a)v_{j},
\\
\widetilde\beta_2^{(2j)}(z,a,b):=& (z-a)u^{*}_{2,j}R_{j}^*(\bar z)
H^{-1}_{2,j}R_{j}(a)u_{2,j}, \\
\widetilde\gamma_2^{(2j)}(z,a,b):=&-(z-a)v^{*}_{j}R_{j}^*(\bar z)
H^{-1}_{2,j}R_{j}(a)v_{j}
\end{align*}\ 
and
 \begin{align*}
\widetilde\delta_2^{(2j)}(z,a,b):=&I_q+(z-a)v^{*}_{j}R_{j}^*(\bar
z)H^{-1}_{2,j}R_{j}(a)u_{2,j},
\end{align*}
 is called the
  second auxiliary matrix of the THMM problem in the
 case of an odd number of moments.
\end{defn}

\begin{defn} Let $(s_k)_{k=0}^{2j}$ be an odd Hausdorff positive
sequence. The $2q\times 2q$ matrix polynomial
\begin{equation}
\widehat U_2^{(2j)}(z,a,b):=\left(\begin{array}{cc}
        \widehat\alpha_2^{(2j)}(z,a,b) &
         \widehat\beta_2^{(2j)}(z,a,b)\\
        \widehat\gamma_2^{(2j)} (z,a,b)&
        \widehat\delta_2^{(2j)}(z,a,b)
        \end{array}\right),\, z\in {\mathbb C}, \quad 0\leq j\leq n-1,
        \label{RM1000}
            \end{equation}
            with
\begin{align*}
 \widehat\alpha_2^{(2j)}(z,a,b):=
 & I_q-(z-a)(u^{*}_{2,j}+zs_0 v_j^*)R_{j}^*(\bar z) H^{-1}_{2,j}R_{j}(a)v_{j}, 
\\
\widehat\beta_2^{(2j)}(z,a,b):=&
(z-a)(s_0+(u^{*}_{2,j}+zs_0v_j^*)R_{j}^*(\bar z)
H^{-1}_{2,j}R_{j}(a)(u_{2,j}+av_js_0)),\\
\widehat\gamma_2^{(2j)}(z,a,b):=&-(z-a)v^{*}_{j}R_{j}^*(\bar z)
H^{-1}_{2,j}R_{j}(a)v_{j}
 \end{align*}
 and
 \begin{align*}
\widehat\delta_2^{(2j)}(z,a,b):=&
 I_q+(z-a)v^{*}_{j}R_{j}^*(\bar z)H^{-1}_{2,j}R_{j}(a)(u_{2,j}+a v_j s_0)
\end{align*}
 is called the second transformed auxiliary matrix of the THMM
  problem in the case of an odd number of moments. The adjective
  transformed in the sequel will be omitted.
\end{defn}
Let
 \begin{equation}
N_{2,j}:=-(b-a)^{-1}v_j^*R_j^*(a)H_{1,j}^{-1}R_j(a)v_j
\label{lei6.15}
\end{equation}
and
\begin{equation}
A_2^{(2j)}:=\left(
       \begin{array}{cc}
         I_q & -as_0 \\
         0_q & I_q \\
       \end{array}
     \right)\left(
       \begin{array}{cc}
         I_q & 0_q \\
         N_{2,j} & I_q \\
       \end{array}
     \right).
 \label{lei6.16}
\end{equation}
\begin{rem}\label{rem00140} Let $(s_j)_{j=0}^{2n}$ be a
Hausdorff positive sequence, and let $U^{(2j)}$, $\widetilde
U_2^{(2j)}$ and $\widehat U_2^{(2j)}$ be as in (\ref{uujd}),
(\ref{RM100}) and (\ref{RM1000}).
 The following equalities are valid:
 \\
 a)
 \begin{equation}
 \widehat U_2^{(2j)}(z)=\left(
                                     \begin{array}{cc}
                                       I_q & zs_0 \\
                                       0_q & I_q \\
                                     \end{array}
                                   \right)\widetilde U_2^{(2j)}(z)\left(
                                     \begin{array}{cc}
                                       I_q & -as_0 \\
                                       0_q & I_q \\
                                     \end{array}
                                   \right) \label{hUh1}
 \end{equation}
 and
 b) \begin{align}
 U^{(2j)}(z)
 =\left(
               \begin{array}{cc}
                 \frac{1}{(z-a)(b-z)}I_q & 0_q \\
                 0_q & I_q \\
               \end{array}
             \right)\widehat U_2^{(2j-2)}(z)\left(
       \begin{array}{cc}
         I_q & 0_q \\
         N_{2,j} & I_q \\
       \end{array}
     \right)\left(
                                   \begin{array}{cc}
                                     (z-a)(b-z)I_q & 0_q \\
                                     0_q & \frac{b-z}{b-a}I_q \\
                                   \end{array}
                                 \right). \label{eq290z}
\end{align}
\end{rem}
\begin{proof} Equalities (\ref{hUh1}) and  (\ref{eq290z}) readily
follow by direct calculations.\end{proof}
\subsection{Orthogonal matrix polynomials on $[a,b]$}
  Let us reproduce some notions on OMP which were introduced in
  \cite{abmad}.
 Let $P$ be a complex $p \times q$~matrix polynomial.
 For all $n\in\N_0$, let
  $$
    Z^{[P]}_{n}
    :=[A_0,A_1,\dotsc,A_n],
$$
where $(A_j)_{j=0}^\infty$ is the unique sequence of complex $p
\times q$~matrices such that for all $z\in\C$ the polynomial $P$
admits the representation $P(z)=\sum_{j=0}^\infty z^jA_j$.
Furthermore, we denote by ${\rm deg}\, P:=\sup\{j\in\N_0:
A_j\neq0_{p\times q}\}$ the \emph{degree of $P$}. Observe that in
the case $P(z)=0_{p\times q}$ for all $z\in\C$ we have thus ${\rm
deg}\, P=-\infty$. If $k:={\rm deg}\, P\geq0$, we refer to $A_k$ as
the \emph{leading coefficient of $P$}.
For all $k\in \N_0$ and all $\kappa\in \N_0$ with $k\leq \kappa$,
let ${\mathbb Z}_{k,\kappa}:=\{n\in \N_0, k\leq n\leq \kappa\}$.
 \begin{defn}\label{def3.2}
    Let $\kappa\in\N_0\cup\{\infty\}$, and let $(s_j)_{j=0}^{2\kappa}$
    be a sequence of complex $q\times q$ matrices.
    A sequence $(P_k)_{k=0}^\kappa$ of complex $q\times q$~matrix polynomials
     is called a \emph{monic left orthogonal
     system of matrix polynomials with respect to
     $(s_j)_{j=0}^{2\kappa}$} if the following three conditions
     are fulfilled:
    \begin{itemize}
        \item[(I)]\label{MLOS.I} $\deg P_k
=k$ for all $k\in {\mathbb Z}_{0,\kappa}$;
        \item[(II)]\label{MLOS.II} $P_k$ has the leading coefficient
        $I_q$ for all $k\in{\mathbb Z}_{0,\kappa}$;
        \item[(III)]\label{MLOS.III} $Z_n^{[P_j]}H_n(Z_n^{[P_k]})^*=
        0_{q\times q}$ for all
         $j,k\in{\mathbb Z}_{0,\kappa}$ with $j\neq k$, where $n:=\max\{j,k\}$.
    \end{itemize}
\end{defn}
\begin{rem} \cite[Remark 3.6]{abmad} Let $n\in \N_0\cup
\{\infty\}$, and let $(s_j)_{j=0}^{2n}$ be a Hausdorff positive
definite sequence, i.e., the corresponding Hankel block matrix
$H_{n}$ is positive definite. Denote by $(P_k)_{k=0}^n$ the monic
left orthogonal system of matrix polynomials with respect to
$(s_j)_{j=0}^{2n}$.
 Let $\sigma$ be a nonnegative Hermitian $q\times q$ measure on $\R$
  satisfying $s_j=\int_{[a,b]}t^j d\sigma(t)$ for $0\leq j\leq2n$.
  Thus,
 \begin{align*}
  \int_{[a,b]}P_jd\sigma P_k^*=\left\{
  \begin{array}{cc}
    \widehat H_j, & \mbox{if}\ \ j=k, \\
    0_q, & \mbox{if}\ \ j\neq k
     \\
  \end{array}
  \right. 
 \end{align*}
 for all $0\leq j,k\leq n$
where $\widehat H_j$ denotes the Schur complement of $H_{j-1}$ in
$H_j$; see (\ref{lkk}).
\end{rem}
\begin{defn} \label{de002}
Let $(s_k)_{k=0}^{2j}$ be a Hausdorff positive definite sequence.
Let
\begin{align}
P_{1,0}(z)&:= I_q, \ \  P_{2,0}(z):=I_q, \ \ Q_{1,0}(z):=0_{q}, \ \
Q_{2,0}(z,a,b):=-(u_{2,0}+z\,s_0), \nonumber 
\\
P_{1,j}(z)&:=(-Y^{*}_{1,j}H^{-1}_{1,j-1},I_q)R_{j}(z)v_{j}, \quad
1\leq j\leq n, \label{92}
\\
P_{2,j}(z,a,b)&:=(-Y^{*}_{2,j}H^{-1}_{2,j-1},I_q)R_{j}(z)v_{j},\quad
1\leq j\leq n-1,\label{92aa}
\\
 Q_{1,j}(z)&:=-(-Y^{*}_{1,j}H^{-1}_{1,j-1},I_q)R_{1,j}(z)
u_{1,j},\quad 1\leq j\leq n \nonumber
 \end{align}
 and
 \begin{align}
Q_{2,j}(z,a,b):=-(-Y^{*}_{2,j}H^{-1}_{2,j-1},I_q)R_{j}(z)
(u_{2,j}+zv_{j}s_{0}), \quad 1\leq j\leq n-1. \label{93aa}
\end{align}
\end{defn}
\begin{defn} \label{def001} Let $K_{k,j}$, $\widetilde u_{k,j}$,
 $\widetilde Y_{k,j}$, for
$k=1,2$, $R_j$ and $v_j$ be as in (\ref{2nmas1}), (\ref{22}),
(\ref{27}), (\ref{22}), (\ref{27}),  (\ref{52}) and (\ref{59}),
respectively.

 Let $(s_k)_{k=0}^{2j+1}$ be a  Hausdorff  positive definite sequence.
Let
\begin{align*}
\Gamma_{1,0}(z)&:=I_q, \, \Gamma_{2,0}(z):=I_q, \ \
\Theta_{1,0}(z):=s_0,\, \Theta_{2,0}(z):=-s_0 
\end{align*}
for all $z\in \C$.
 For $k\in \{1,2\}$ and $1\leq j\leq n$,
 define
\begin{align}
\Gamma_{1,j}(z,b) &:=(-\widetilde
Y^{*}_{1,j}K^{-1}_{1,j-1},I_q)R_{j}(z)v_{j},\label{38-1}
\\
 \Gamma_{2,j}(z,a) &:=(-\widetilde Y^{*}_{2,j}K^{-1}_{2,j-1},I_q)
 R_{j}(z)v_{j},\label{38-2}\\
 \Theta_{1,j}(z,b) &:= (-\widetilde Y^{*}_{1,j}K^{-1}_{1,j-1},I_q)
 R_{j}(z)\widetilde u_{1,j}\label{39-1}
 \end{align}
 and
 \begin{align}
 \Theta_{2,j}(z,a)&:= (-\widetilde Y^{*}_{2,j}K^{-1}_{2,j-1},I_q)
 R_{j}(z)\widetilde u_{2,j}, \label{39-2}
\end{align}
for all $z\in \C$.
\end{defn}
We  usually omit the dependence of the polynomials $P_{k,j}$,
$Q_{k,j}$, $\Gamma_{k,j}$ and $\Theta_{k,j}$ for $k=1,2$ on the
parameters $a$ and $b$.

In \cite{abpol}, (resp. \cite{thi}) it was proved that polynomials
$P_{k,j}$ (resp. $\Gamma_{k,j}$) for $k=1,2$ are in fact OMP on
$[a,b]$. In \cite{abKN} some properties of second kind polynomials
$Q_{k,j}$ and $\Theta_{k,j}$ for $k=1,2$ were discussed. In
\cite{abmad} explicit interrelations between $P_{k,j}$,
$\Gamma_{k,j}$ and their second kind polynomials were studied.

For the sake of completeness in the following Remark, we reproduce
 explicit interrelations between the matrices
  $\widehat H_{k,j}$, $ \widehat K_{k,j}$
  and the polynomials $P_{1,j}$, $Q_{2,j}$, $\Gamma_{1,j}$, $\Theta_{2,j}$
 considered in \cite[Corollary 3.4]{abKN} and \cite[Corollary 3.10]{abKN}.
 \begin{rem} Let $\widehat H_{k,j}$, $ \widehat K_{k,j}$, for
 $k=1,2$, $P_{1,j}$, $Q_{2,j}$, $\Gamma_{1,j}$ and $\Theta_{2,j}$
 be as in (\ref{lkk}), (\ref{lkk2}), (\ref{lkk3}), (\ref{lkk4})
and  Definitions \ref{de002} and \ref{def001}, respectively.
 The  following equalities then hold:
\begin{align}
  \widehat H_{1,j}=&-P_{1,j}(a)\Theta_{2,j}^*(a), \ \ 
  \widehat H_{2,j}=-Q_{2,j}(a)\Gamma_{1,j+1}^{*}(a),\label{pqHHa}\\
  \widehat K_{1,j}=&\Gamma_{1,j}(a)Q_{2,j}^*(a), \ \ 
  \widehat K_{2,j}=\Theta_{2,j}(a)P_{1,j+1}^{*}(a). \label{pqHHa11}
\end{align}
 \end{rem}
\section{Algebraic identities}
 In this section we will single out essential identities concerning
  the block matrices introduced in Section \ref{sec002}:
\begin{gather*}
 L_{1,n}:=\left(
          \delta_{j,k+1}I_q
        \right)_{{\tiny\begin{array}{c}
                   j=0,\ldots,n \\
                   k=0,\ldots, n-1
                 \end{array}}
        }\ \ \ \mbox{and} \ \ \
        L_{2,n}:=\left(
          \delta_{j,k}I_q
        \right)_{{\tiny\begin{array}{c}
                   j=0,\ldots,n \\
                   k=0,\ldots, n-1
                 \end{array}}
        },
\end{gather*}
 where $\delta_{j,k}$ is the Kronecker symbol with $\delta_{j,k}:=1$
if $j=k$ and $\delta_{j,k}:=0$ if $j\neq k$.

 Let
\begin{equation}
 \Xi_{1,j}^K:=\left( \begin{matrix}
              -K_{1,j-1}^{-1} \widetilde{Y}_{1,j} \\ I_q \end{matrix}
              \right).\label{kk12aa}
 \end{equation}
\begin{rem}\label{leiprem001}
The following identities are valid:
 \begin{align}
  &v_{j-1}-L_{2,j}^*v_j=0, \label{leivj01}\\
  &\widetilde u_{1,j-1}-L_{2,j}^*\widetilde u_{1,j}=0, \label{leitu1j01}
   \\
  &L_{2,j}-R_j^{*^{-1}}(\bar z)L_{2,j}R_{j-1}^{*^{-1}}(\bar z)=0,\label{leiLRL2}\\
  &L_{2,j}L_{1,j}^*-T_j^*=0, \label{leiLLTj}
 \\
  &H_{1,j}T_j^*-T_jH_{1,j}-\widetilde u_{1,j}v_j^*+v_j\widetilde
  u_{1,j}=0 \label{leiA2.2}\\
  &R_{j-1}^{-1}(a)K_{1,j-1}L_{1,j}^*
  +L_{2,j}^*K_{1,j}T_j^*-L_{2,j}^*T_jK_{1,j}+L_{2,j}^*T_jK_{2,j}R_j^{*^{-1}}(a)=0,
   \label{leipKLLK1}\\
   &T_jK_{1,j}\Xi_{1,j}^K=0. \label{leiK100K}
 \end{align}
\end{rem}
\begin{proof}
 Equalities  (\ref{leivj01}), (\ref{leitu1j01}), (\ref{leiLRL2}), (\ref{leiLLTj})
 are proved by direct calculations. Identity (\ref{leiA2.2}) was
 considered in \cite[Proposition 2.1]{abdon1}. Identities
 (\ref{leipKLLK1}) and (\ref{leiK100K})
 follow by a straightforward calculation.
\end{proof}
\begin{rem}\label{leiprem002q}
The following identities are valid:
 \begin{align}
 &u_{1,j}^*+v_j^*H_{1,j}T_j^*=0,\label{leiHH1}\\
 &v_{j}^*H_{1,j}-v_{j+1}^*H_{1,j+1}L_{2,j+1}=0, \label{leiHH11}\\
 &T_{j+1}^*L_{1,j+1}+(T_{j+1}T_{j+1}^*-I)L_{2,j+1}-L_{2,j+1}L_{1,j}L_{1,j}^*=0,
 \label{leiHH22}
 \\
&T_{j+1}^*L_{1,j+1}T_j^*-T_{j+1}^*L_{2,j+1}L_{1,j}L_{1,j}^*=0,
 \label{leiHH33}
 \\
 &T_{j+1}^*L_{2,j+1}-L_{2,j+1}L_{2,j}L_{1,j}^*=0,
 \label{leiHH44}
 \\
 &T_{j+1}^*L_{2,j+1}T_j^*-T_{j+1}^*L_{2,j+1}L_{2,j}L_{1,j}^*=0,
 \label{leiHH55}
 \\
 &(I-zT_{j+1}^*)L_{2,j+1}(T_jT_j^*-I)-(T_{j+1}T_{j+1}^*-I)L_{2,j+1}
 =0,\label{leiHH2}\\
&(I-z T_{j+1}^*)L_{2,j+1}\left(I+(z-a)T_j^*R_j^*(\bar z)\right)
-(I-aT_{j+1}^*)L_{2,j+1}=0,\label{leiHH3}
\\
 & v_jv_{j+2}^*H_{1,j+2}L_{1,j+2}+L_{2,j+1}^*H_{1,j+1}
+L_{1,j+1}^*T_{j+1}H_{1,j+1}=0,\label{leiHH41}
 \\
 &
v_jv_{j+2}^*H_{1,j+2}L_{2,j+2} +L_{2,j+1}^*T_{j+1}\widetilde
H_{2,j+1}-L_{1,j+1}^* \widetilde H_{1,j+1}=0,\label{leiHH42}
 \\
 & -T_jL_{1,j+1}^*H_{1,j+1}+L_{2,j+1}T_{1,j+1}\widetilde H_{1,j+1}+
  T_jL_{2,j+1}^*T_{j+1}\widetilde H_{2,j+1}=0, \label{lieHH43}\\
  & -L_{2,j+1}T_{j+1}\widetilde H_{0,j+1}+ T_j
  L_{2,j+1}^*H_{1,j+1}=0. \label{liehh44}
 \end{align}
\end{rem}
\begin{proof}
 Identities (\ref{leiHH1})--(\ref{liehh44})
  follow by a straightforward calculation.
\end{proof}
Denote
 \begin{equation}
 \Xi_{2,j}^H:=\left( \begin{array}{c}
              -H_{2,j-1}^{-1} Y_{2,j} \\ I_q
              \end{array}
              \right).\label{hh2n12}
 \end{equation}
\begin{prop} The following identities are valid:
\begin{align}
&-T_{j+1}^*(L_{1,j+1}-bL_{2,j+1})-(T_{j+1}T_{j+1}^*-I)L_{2,j+1}R_j^*(a)
\nonumber\\
&+(I-a
T_{j+1}^*)L_{2,j+1}(L_{1,j}-bL_{2,j})L_{1,j}^*R_j^*(a)=0,\label{lieHH4}\\
&-R_j(a)v_jv_{j+2}^*H_{1,j+2}(L_{1,j+2}-bL_{2,j+2})
+(L_{1,j+1}^*-bL_{2,j+1}^*)
\nonumber\\
&+L_{2,j+1}R_{j+1}(a)T_{j+1}H_{2,j+1}=0,\label{leiHH5}\\
   &T_jH_{2,j}\Xi_{2,j}^H=0. \label{leiH100H}
\end{align}
\end{prop}
\begin{proof}
 Identity (\ref{leiHH44}) follows from (\ref{leiHH22})--(\ref{leiHH55}).
 Identity (\ref{lieHH4})  follows  (\ref{leiHH41})--(\ref{leiHH44}).
 We prove equality (\ref{leiH100H}). Let $\lambda_j:=(0_q,0_q,\ldots,0_q, I_q)$
 be a $q\times j q$ matrix. Thus $T_j=\left(
                                        \begin{array}{cc}
                                          T_{j-1} & 0_{jq}\times q \\
                                          \lambda_j & 0_q \\
                                        \end{array}
                                      \right)$.
 By using the last equality and equality $H_{2,j}=\left(
                                                  \begin{array}{cc}
                                                    H_{2,j-1} & Y_{2,j} \\
                                                    Y_{2,j}^* & \widehat s_{2j} \\
                                                  \end{array}
                                                \right),
  $ we have
  \begin{align*}
  T_jH_{2,j}\Xi_{2,j}^H=&\left(
   \begin{array}{cc}
   T_{j-1} & 0_{jq}\times q \\
   \lambda_j & 0_q \\
  \end{array}
 \right)\left(
 \begin{array}{cc}
  H_{2,j-1} & Y_{2,j} \\
  Y_{2,j}^* & \widehat s_{2j} \\
  \end{array}
  \right)\left(
 \begin{array}{c}
 -H_{2,j-1}^{-1}Y_{2,j} \\I_q \end{array}
  \right)\\
  =
  &\left(
  \begin{array}{cc}
  T_{j-1}H_{2,j-1} & T_{j-1}Y_{2,j}\\
  \lambda_jH_{2,j-1} & \lambda_j Y_{2,j}
  \end{array}\right)
   \left(  \begin{array}{c}
  -H_{2,j-1}^{-1}Y_{2,j} \\   I_q   \end{array}
  \right)=\left(
            \begin{array}{c}
              0_{jq\times q} \\
              0_q
            \end{array}
          \right).
\end{align*}
\end{proof}

\section{The Blaschke--Potapov factors of the auxiliary matrices}
 In this section we obtain a multiplicative representation (\ref{uuu2nm1}),
  (\ref{uuu2n})  of the
 second auxiliary matrices
  $\widetilde U_2^{(2n+1)}$  and $\widehat U_2^{(2n)}$
  via the Blaschke--Potapov factors $d^{(2j+1)}$ and $d^{(2j)}$
   defined in (\ref{leipg0})--(\ref{leipfj}).

Since the matrices $H_{2,j}$ and
 $K_{1,j}$ are positive definite matrices for
$0\leq j\leq
 n-1$ and  $0\leq j\leq
 n$, respectively, their inverses can be written as
 \begin{align}
 H^{-1}_{2,j}=&
                \left(\begin{array}{cc}
          H^{-1}_{2,j-1} & 0_{jq\times q}\\
          0_{q\times jq} & 0_{q}\\
         \end{array}\right)+ \left(\begin{array}{c}
                                 -H^{-1}_{2,j-1}Y_{2,j}\\
                                                 I_q\\
                              \end{array}\right)
\widehat H_{2,j}^{-1}(-Y^{*}_{2,j}H^{-1}_{2,j-1},I_q)\label{32aa}
\end{align}
and
\begin{align}
 K^{-1}_{1,j}=&
                \left(\begin{array}{cc}
          K^{-1}_{1,j-1} & 0_{jq\times q}\\
          0_{q\times jq} & 0_{q\times q}\\
         \end{array}\right)+ \left(\begin{array}{c}
                                 -K^{-1}_{2,j-1}\widetilde Y_{1,j}\\
                                                 I_q\\
                              \end{array}\right)
\widehat K_{1,j}^{-1}(-\widetilde
Y^{*}_{1,j}K^{-1}_{1,j-1},I_q).\label{32aa-2np1}
\end{align}

\begin{prop}\label{leipprop1aa} Let the polynomials $P_{2,j}$ and
$Q_{2,j}$ be as in Definition \ref{de002} and $\widehat H_{2,j}$ be
defined as in (\ref{lkk2}). The block elements of
 the matrix $\widehat U_2^{(2j)}(z)$ defined by  (\ref{RM1000})
 can be written in the form
\begin{align*}
 \widehat \alpha_2^{(2j)}(z)&=
  \widehat \alpha_2^{(2j-2)}(z)+(z-a)  Q_{2,j}^*(\bar z)
  \widehat H_{2,j}^{-1}
   P_{2,j}(a),
\\
\widehat \beta_2^{(2j)}(z)&= \widehat \beta_2^{(2j-2)}(z)+(z-a)
 Q_{2,j}^*(\bar z)\widehat H_{2,j}^{-1}  Q_{2,j}(a),
\\
\widehat \gamma_2^{(2j)}(z)&= \widehat \gamma_2^{(2j-2)}(z)-
(z-a)P_{2,j}^*(\bar z)\widehat H_{2,j}^{-1} P_{2,j}(a)
\end{align*}
and
 $$ 
\widehat \delta_2^{(2j)}(z)= \widehat \delta_2^{(2j-2)}(z)-
(z-a)P_{2,j}^*(\bar z)\widehat H_{2,j}^{-1}  Q_{2,j}(a).
$$
\end{prop}
\begin{proof} Use (\ref{59}), (\ref{32aa})
and
\begin{equation}
R_{j}(z)=\left(
           \begin{array}{c|c}
             R_{j-1}(z) & 0_{jq\times q} \\
             \hline
             (z^j I_q, z^{j-1} I_q,\ldots, zI_q) & I_q \\
           \end{array}
         \right), \quad
         u_{2,j}=\left(
                   \begin{array}{c}
                     u_{2,j-1} \\
                     -\widehat s_{j-1}
                   \end{array}
                 \right)\label{rruub}
\end{equation}
for $j\geq 2$.
\end{proof}

\begin{prop}\label{leipprop1bb} Let the polynomials $\Theta_{1,j}$ and
 $\Gamma_{1,j}$ be as in Definition \ref{def001}. Then
the block elements of the matrix $\widetilde U_2^{(2j+1)}(z)$
defined by (\ref{leipRMe22})
  can be written in the form
\begin{align*}
 \widetilde \alpha_2^{(2j+1)}(z)&=
  \widetilde \alpha_2^{(2j-1)}(z)-(z-a)
   \Theta_{1,j}^*(\bar z)\widehat K_{1,j}^{-1} \Gamma_{1,j}(a),
\\
\widetilde \beta_2^{(2j+1)}(z)&= \widetilde
\beta_2^{(2j-1)}(z)+(z-a) \Theta_{1,j}^*(\bar z)\widehat
K_{1,j}^{-1} \Theta_{1,j}(a),
\\
\widetilde \gamma_2^{(2j+1)}(z)&= \widetilde \gamma_2^{(2j-1)}(z)-
(z-a)\Gamma_{1,j}^*(\bar z)\widehat K_{1,j}^{-1} \Gamma_{1,j}(a)
\end{align*}
and
 $$
\widetilde \delta_2^{(2j+1)}(z)= \widetilde \delta_2^{(2j-1)}(z)+
(z-a)\Gamma_{1,j}^*(\bar z)\widehat K_{1,j}^{-1} \Theta_{1,j}(a).
$$
\end{prop}
\begin{proof}
Use (\ref{59}), (\ref{32aa-2np1}), the first equality of
(\ref{rruub}) and equality
\begin{equation}
        \widetilde u_{1,j}=\left(
                   \begin{array}{c}
                     \widetilde u_{1,j-1} \\
                     -bs_{j-1}+s_j
                   \end{array}\right)
                  \label{uttj1}
\end{equation}
 for $j\geq2$.
\end{proof}

\begin{defn} \label{bjdefleip} Let $\widehat H_{2,j}$, $\widehat
K_{1,j}$, $P_{2,j}$, $Q_{2,j}$,
 $\Theta_{1,j}$ and $\Gamma_{1,j}$ be as in  (\ref{lkk2}), (\ref{lkk3}),
  and Definitions \ref{de002}, \ref{def001}, respectively. Define
\begin{align}
 d^{(0)}(z):=&\left(
       \begin{array}{cc}
         I_q
         & (z-a)s_0 \\
         0_q
         & I_q
       \end{array}
     \right),\label{leipg0}
 \\
 d^{(2j+2)}(z):=&\left(
       \begin{array}{cc}
         I_q+(z-a)Q_{2,j}^*(a)\widehat H_{2,j}^{-1} P_{2,j}(a)
         & (z-a)Q_{2,j}^*(a)\widehat H_{2,j}^{-1}  Q_{2,j}(a) \\
         -(z-a)P_{2,j}^*(a)\widehat H_{2,j}^{-1} P_{2,j}(a)
         & I_q-(z-a)P_{2,j}^*(a)\widehat H_{2,j}^{-1} Q_{2,j}(a)
       \end{array}
     \right) \label{leipgj}
 \end{align}
 for  $0\leq j\leq n-1$, and
 \begin{equation}
 d^{(2j+1)}(z):=\left(
       \begin{array}{cc}
         I_q-(z-a)\Theta_{1,j}^*(a)\widehat K_{1,j}^{-1} \Gamma_{1,j}(a)
         & (z-a)\Theta_{1,j}^*(a)\widehat K_{1,j}^{-1} \Theta_{1,j}(a) \\
         -(z-a)\Gamma_{1,j}^*(a)\widehat K_{1,j}^{-1} \Gamma_{1,j}(a)
         & I_q+(z-a)\Gamma_{1,j}^*(a)\widehat K_{1,j}^{-1} \Theta_{1,j}(a)
       \end{array}
     \right) \label{leipfj}
 \end{equation}
 for  $0\leq j\leq n$.

The matrix function $d^{(2j)}$ (resp. $d^{(2j+1)}$) is called the
Blaschke--Potapov factor of the auxiliary matrix $\widehat
U_2^{(2k)}$ (resp. $\widetilde U_2^{(2k+1)}$).
\end{defn}

\begin{thr} \label{mthaaleip} Let the matrix $\widehat U_2^{(2j)}$
(resp. $\widetilde U_2^{(2j+1)}$) be as in (\ref{RM1000}) (resp.
(\ref{leipRMe22})). Let $d^{(j)}$ be defined as in
(\ref{leipg0})-(\ref{leipfj}), then
 \begin{align}
 \widehat U_2^{(0)}(z)=&d^{(0)}(z)d^{(2)}(z),
 \quad \widetilde U_2^{(1)}(z)=d^{(1)}(z), \label{U20d}\\
  \widehat U_2^{(2j)}(z)=& \widehat U_2^{(2j-2)}(z)d^{(2j+2)}(z),\, z\in\C,
  \quad 1\leq j\leq n-1
   \label{leip12n}
\end{align}
and
\begin{align}
 \widetilde U_2^{(2j+1)}(z)=& \widetilde U_2^{(2j-1)}(z)d^{(2j+1)}(z),
  \, z\in\C, \quad 1\leq j\leq n.
   \label{leip1aa}
 \end{align}
\end{thr}
\begin{proof}
Equality (\ref{U20d}) readily follows from direct calculation. Now
we demonstrate (\ref{leip12n}). Denote
 \begin{align}
  G_{j}^{11}(a)&:=Q_{2,j}^*(a)\widehat H_{2,j}^{-1} P_{2,j}(a),
  \label{leipG11}\\
  G_{j}^{12}(a)&:= Q_{2,j}^*(a)\widehat H_{2,j}^{-1}
    Q_{2,j}(a),\label{leipG12}\\
  G_{j}^{21}(a)&:=P_{2,j}^*(a)\widehat H_{2,j}^{-1} P_{2,j}(a), \label{leipG21}
  \end{align}
  and
  \begin{equation}
  G_{j}^{22}(a):=P_{2,j}^*(a)\widehat H_{2,j}^{-1}  Q_{2,j}(a)\label{leipG22}
 \end{equation}
 for $1\leq j\leq n-1$.

 Now we prove equality (\ref{leip12n}).
  By using (\ref{leipg0}), (\ref{leipgj}), (\ref{leipG11}),
    (\ref{leipG12}), (\ref{leipG21}) and
     (\ref{leipG22}),
     Eq. (\ref{leip12n})
   can be written in the equivalent form
 \begin{align}
 &\left(
   \begin{array}{cc}
     \widehat \alpha_2^{(2j)}(z) &\widehat \beta_2^{(2j)}(z)\\
     \widehat \gamma_2^{(2j)}(z) & \widehat \delta_2^{(2j)}(z)
   \end{array}
 \right)\nonumber\\
 &-\left(
   \begin{array}{cc}
     \widehat \alpha_2^{(2j-2)}(z) &\widehat \beta_2^{(2j-2)}(z)\\
     \widehat \gamma_2^{(2j-2)}(z) & \widehat \delta_2^{(2j-2)}(z)
   \end{array}
 \right)\left(
       \begin{array}{cc}
         I_q+(z-a)G_{j}^{11}(a)
         & (z-a)G_{j}^{12}(a) \\
         -(z-a)G_{j}^{21}(a)
         & I_q-(z-a)G_{j}^{22}(a)
       \end{array}
     \right)=0. \label{meq2naa}
 \end{align}
The left-hand side of (\ref{meq2naa}) is equivalent to the following
four equalities:
\begin{align*}
\Upsilon_{11,j} :=&\widehat \alpha_2^{(2j)}(z) -\widehat
\alpha_2^{(2j-2)}(z)+(z-a) \left(-\widehat
\alpha_2^{(2j-2)}(z)G_{j}^{11}(a)+
\beta_2^{(2j-2)}(z)G_{j}^{21}(a)\right), 
\\
\Upsilon_{12,j}:= &\widehat \beta_2^{(2j)}(z) -\widehat
\beta_2^{(2j-2)}(z)-(z-a) \left(\widehat
\alpha_2^{(2j-2)}(z)G_{j}^{12}(a)-
\beta_2^{(2j-2)}(z)G_{j}^{22}(a)\right), 
\\
\Upsilon_{21,j}:=&\widehat \gamma_2^{(2j+1)}(z) -\widehat
\gamma_2^{(2j-2)}(z)+(z-a) \left( -\widehat
\gamma_2^{(2j-2)}(z)G_{j}^{11}(a)+
\delta_2^{(2j-2)}(z)G_{j}^{21}(a)\right) 
\end{align*}
and
 \begin{align*}
\Upsilon_{22,j}:=\widehat \delta_2^{(2j+1)}(z) -\widehat
\delta_2^{(2j-2)}(z)-(z-a) \left(\widehat
\gamma_2^{(2j-2)}(z)G_{j}^{12}(a)-
\delta_2^{(2j-2)}(z)G_{j}^{22}(a)\right). 
  \end{align*}
By taking into account (\ref{leipG11}) and (\ref{leipG21}), we have
  \begin{align*}
\Upsilon_{11,j}=&(z-a)\widetilde \Upsilon_{1,j}\widehat
H_{2,j}^{-1}P_{2,j}(a),\quad \Upsilon_{12,j}=(z-a)\widetilde
\Upsilon_{1,j}\widehat H_{2,j}^{-1}
 Q_{2,j}(a),\\ 
\Upsilon_{21,j}=&(z-a)\widetilde \Upsilon_{2,j}\widehat
H_{2,j}^{-1}P_{2,j}(a),
 \quad
 \Upsilon_{22,j}=(z-a)\widetilde \Upsilon_{2,j}\widehat
 H_{2,j}^{-1} Q_{2,j}(a) 
  \end{align*}
where
 \begin{align*}
\widetilde \Upsilon_{1,j+1}:=&Q^*_{2,j+1}(\bar z)- Q^*_{2,j+1}(a)
 +(z-a)(u^{*}_{2,j}+zs_0v_j^*)R_{j}^*(\bar z)
 H_{2,j}^{-1}R_{j}(a)v_{j} Q_{2,j+1}^*(a)\nonumber\\
  &+(z-a)(s_0+ (u^{*}_{2,j}+zs_0v_j^*)R_{j}^*(\bar z)H_{2,j}^{-1}R_{j}(a)
  (u_{2,j}+av_js_0))P_{2,j+1}^*(a), 
  \\
\widetilde \Upsilon_{2,j+1}:=&-P^*_{2,j+1}(\bar z)+P^*_{2,j+1}(a)
 +(z-a)v_{j}^*R_{j}^*(\bar z)
 H_{2,j}^{-1}R_{j}(a)v_{j} Q_{2,j+1}^*(a)\nonumber\\
  &+(z-a) v_{j}^*R_{j}^*(\bar z)H_{2,j}^{-1}R_{j}(a)
  (u_{2,j}+av_j s_0)P_{2,j+1}^*(a). 
 \end{align*}
  Now we verify that
  \begin{equation}
\widetilde \Upsilon_{\ell,j}=0, \quad \ell ,k\in\{1,2\}, \quad 1\leq
j\leq n-1. \label{leipyyjk-2n}
  \end{equation}
By using (\ref{93aa}), (\ref{92aa}) and (\ref{hh2n12}), we have
\begin{align*}
\widetilde \Upsilon_{1,j+1} &= \left(
 -(u_{2,j+1}+zs_0v_{j+1}^*)R_{j+1}^*(\bar z)+
 (u_{2,j+1}+as_0v_{j+1}^*)R_{j+1}^*(a)+
 \right.\\
 &-(z-a)(u_{2,j}+zs_0v_{j}^*)R_{j}^*(\bar z)H_{2,j}^{-1}R_j(a)
 (u_{2,j+1}^*+a s_0 v_{j+1}^*)R_{j+1}^*(a)\\
 &+(z-a)
 \left(s_0+
  (u_{2,j}^*+z s_0v_j^*)R_j^*(\bar z)H_{2,j}^{-1}
  R_j(a)(u_{2,j}+a v_j s_0)\right)\\
  &\cdot \left.v_{j+1}^*R_{j+1}^*(a)\right)\Xi_{2,j}^H
  \\
  =&(z-a)v_{j+2}^*H_{1,j+2}R_{j+2}^*(\bar z)
  \left(
   -T_{j+2}^*(L_{1,j+2}-bL_{2,j+2})\right.\\
   &-
   \left((I-z T_{j+2}^*)L_{2,j+2}(L_{1,j+2}-bL_{2,j+1})+(z-a)
    (I-z T_{j+1}^*)L_{2,j+2}\right.\\
&\cdot    R_{j+1}^*(\bar
z)T_{j+1}^*(L_{1,j+1}-bL_{2,j+1}))H_{2,j}^{-1}R_j(a)v_jv_{j+2}^*H_{1,j+2}
  \\
 &\cdot (L_{1,j+1}-bL_{2,j+1})-
  (I-z T_{j+1}^*)L_{2,j+2}(T_{j+1}T_{j+1}^*-I)R_{j+1}^*(a)\\
  &+
  \left((I-z T_{j+2}^*)L_{2,j+2}(L_{1,j+1}-bL_{2,j+1})+(z-a)(I-z
  T_{j+2}^*)
  L_{2,j+2}\right.\\
  &\left.\cdot R_{j+1}^*(\bar z)T_{j+1}^*  (L_{1,j+1}-bL_{2,j+1})\right)\\
  &\cdot \left.\left(L_{1,j+1}R_{j+1}^*(a)+H_{2,j}^{-1}
  (L_{1,j+1}^*-bL_{2,j+1}^*)H_{1,j+1}
  \right)\right)\Xi_{2,j}^H
  \\
  =&(z-a)v_{j+2}^*H_{1,j+2}R_{j+2}^*(\bar z)\left(
   -T_{j+2}^*(L_{1,j+1}-bL_{2,j+1})-(T_{j+2}T_{j+2}^*-I)\right.\\
   &\cdot L_{2,j+2}R_{j+1}^*(a)   +(I-a
   T_{j+1}^*)L_{2,j+2}(L_{1,j+1}-bL_{2,j+1})L_{1,j+1}^*R_{j+1}^*(a)\\
   &-(I-a T_{j+2}^*)L_{2,j+2}(L_{1,j+1}-bL_{2,j+1})H_{2,j}^{-1}R_j(a)v_jv_{j+2}^*
   H_{1,j+2}\\
   &\cdot (L_{1,j+1}-bL_{2,j+1})
   +(I-a T_{j+2}^*)L_{2,j+2}(L_{1,j+1}-bL_{2,j+1})
   H_{2,j}^{-1}\\
   &\cdot \left.
   (L_{1,j+1}^*-bL_{2,j+1}^*)H_{1,j+1}\right)\Xi_{2,j}^H\\
=&-(z-a)v_{j+2}^*H_{1,j+2}R_{j+2}^*(\bar z)
 (I-a
 T_{j+2}^*)L_{2,j+2}(L_{1,j+1}-bL_{2,j+1})H_{2,j}^{-1}\\
 &\cdot L_{2,j+1}^*T_{j+1}H_{2,j+1}
\Xi_{2,j}^H\\
=&0.
  \end{align*}
In the second equality we used (\ref{leiHH1}) and (\ref{leiHH11}).
 In the third equality we employed (\ref{leiHH2}) and (\ref{leiHH3}).
 The
penultimate  equality follows from (\ref{lieHH4}) and
(\ref{leiHH5}). The last equality follows from identity
(\ref{leiH100H}).
 Equality (\ref{leipyyjk-2n}) for $\ell=2$ is
proved by using (\ref{92aa}), (\ref{93aa}), (\ref{leivj01}),
(\ref{leiLLTj}), (\ref{leiHH5}) and (\ref{leiH100H}).

To prove (\ref{leip1aa}), we used the following equalities:
\begin{align}
\Theta^*_{1,j}&(\bar z)-\Theta^*_{1,j}(a)
 +(z-a)\widetilde u_{1,j-1}^*R_{j-1}^*(\bar z)
 K_{1,j-1}^{-1}R_{j-1}(a)v_{j-1}\Theta_{1,j}^*(a)\nonumber\\
  &-(z-a)\widetilde u_{1,j-1}^*R_{j-1}^*(\bar z)K_{1,j-1}^{-1}R_{j-1}(a)\widetilde
  u_{1,j-1}\Gamma_{1,j}^*(a)=0\label{leipyk1}
  \end{align}
  and
  \begin{align}
\Gamma^*_{1,j}&(\bar z)-\Gamma^*_{1,j}(a)
 +(z-a)\widetilde v_{j-1}^*R_{j-1}^*(\bar z)
 K_{1,j-1}^{-1}R_{j-1}(a)v_{j-1}\Theta_{1,j}^*(a)\nonumber\\
  &-(z-a)\widetilde v_{j-1}^*R_{j-1}^*(\bar z)K_{1,j-1}^{-1}R_{j-1}(a)\widetilde
  u_{1,j-1}\Gamma_{1,j}^*(a)=0.\label{leipyk2}
 \end{align}
 In turn (\ref{leipyk1}), (\ref{leipyk2}) are demonstrated by using
(\ref{38-1}), (\ref{39-1}),
 (\ref{leitu1j01}), (\ref{kk12aa}), (\ref{leiLLTj}),
 (\ref{leivj01}), (\ref{leitu1j01}), (\ref{leiA2.2}), (\ref{leipKLLK1})
 and (\ref{leiK100K}).  The theorem is proved.
\end{proof}

\begin{sleds}\label{cor-2np1} %
Let the auxiliary matrices  $\widetilde U_2^{(2j+1)}$ and $\widehat
U_2^{(2j)}$ defined as in  (\ref{leipRMe22}) and (\ref{RM1000}).
Furthermore, let $d^{(j)}$ be as in (\ref{leipg0})--(\ref{leipfj}).
Thus for $1\leq j\leq n$ the admit a  Blaschke--Potapov
multiplicative representation (\ref{uuu2nm1}) and (\ref{uuu2n}).
\end{sleds}
The proof follows immediately  from Theorem \ref{mthaaleip}.
\section{Representation of the RM via DSM parameters
of the THMM problem}
 In this section we obtain a multiplicative representation of the
 RM of the THMM problem in terms of OMPs on $[a,b]$ (more information
 can be found in Definitions \ref{de002} and \ref{def001}) and DSM parameters.
\begin{defn} \label{st2ndef-2np1} Let $a$ and $b$ be real numbers such that $a<b$.
  Let $H_{2,j}$, $K_{1,j}$,
 $R_j$, $v_j$, $u_{2,j}$ be defined by (\ref{63}), (\ref{2nmas1}),
  (\ref{52}), (\ref{59}), (\ref{22}) and (\ref{uuu001}), respectively.
 Furthermore, let $H_{2,j}$, $K_{1,j}$ be positive definite
 matrices.
For $1\leq j\leq n-1$ denote by
\begin{align}
 \rb_0:=&s_0,\nonumber \\
 \rb_j(a,b):=&s_0+(u^*_{2,j}+as_0v_j^*)R_j^*(a)H_{2,j}^{-1}R_j(a)
(u_{2,j}+av_j s_0), \label{ll-2np1}
\\
\tb_0(b):=&v_0^*R_0^*(a)K_{1,0}^{-1}R_0(a)v_0,\quad
 \tb_j(a,b):=
v_j^*R_j^*(a)K_{1,j}^{-1}R_j(a)v_j,  \label{ttt0j}
 \\
 \lb_{-1}:=&s_0,\quad \lb_0(a,b):=(u_{2,0}^*+a s_0 v_0^*)H_{2,0}^{-1}(u_{2,0}+av_0 s_0),
 \label{mmm01}\\
 \lb_j(a,b):=&(u^*_{2,j}+a s_0 v_j^*)R_j^*(a)H_{2,j}^{-1}R_j(a)
(u_{2,j}+av_j s_0)\nonumber\\
&-(u^*_{2,j-1}+a s_0 v_{j-1}^*)R_{j-1}^*(a)H_{2,j-1}^{-1}R_{j-1}(a)
 (u_{2,j-1}+av_{j-1} s_0),\label{mmmj}
 \end{align}
 and
 \begin{align}
 \mb_0(b):=&\tb_0(b),\,
 \mb_j(a,b):=v_{j}^*R_{j}^*(a)K_{1,j}^{-1}R_{j}(a)v_{j}-
 v_{j-1}^*R_{j-1}^*(a)K_{1,j-1}^{-1}R_{j-1}(a)v_{j-1}
 \label{lllj}
\end{align}
for $1\leq j\leq n.$
 The matrices  $\lb_j(a,b)$ and $\mb_j(a,b)$ are
  called the second type Dyukarev--Stieltjes matrix parameters of the
   THMM problem.
\end{defn}
 Below we shall usually omit the dependence on $a$
and $b$ of the matrices (\ref{ll-2np1})-(\ref{lllj}).

 Note that from (\ref{32aa}), (\ref{32aa-2np1}), (\ref{rruub}),
  (\ref{uttj1}), (\ref{93aa}), (\ref{38-1}), (\ref{39-1}) and (\ref{92aa})
   the following identities are valid:
\begin{align}
\lb_j=&Q_{2,j}^*(a)\widehat H_{2,j}^{-1} Q_{2,j}(a),\quad
\mb_j=\Gamma_{1,j}^*(a)\widehat K_{1,j}^{-1}\Gamma_{1,j}(a),
\label{mm01-2np1}\\
 \rb_j=&\Gamma_{1,j}^{-1}(a)\Theta_{1,j}(a),
 \qquad \tb_j=Q_{2,j}^{-1}(a)P_{2,j}(a). \label{tQP2j}
\end{align}
\begin{rem} Let $\rb_j$, $\tb_j$, $\lb_j$ and $\mb_j$ be as in
(\ref{ll-2np1})-(\ref{lllj}). Thus, the following equalities hold:
\begin{align}
 \lb_j=&\rb_{j+1}-\rb_{j},\quad j\geq 0,  \label{mm100a}\\
 \mb_j=&\tb_j-\tb_{j-1}, \quad j\geq 1. \label{mm200a}
\end{align}
Moreover, the matrices $\lb_j$ and $\mb_j$ are positive definite
matrices.
\end{rem}
\begin{proof} 
 Equalities (\ref{mm100a})-(\ref{mm200a})
follow by direct calculation from (\ref{ll-2np1})-(\ref{lllj}).\\
 By the second equality of (\ref{pqHHa}) and the fact that $\widehat K_{1,j}$
 and $\widehat H_{2,j}$ are positive definite matrices we obtain
 that $\lb_j$ and $\mb_j$ also are.
\end{proof}
The following theorem shows an explicit representation between the
 Blaschke--Potapov factors $d^{(j)}$ and the matrices $\rb_j$,
  $\tb_j$, $\lb_j$ and $\mb_j$.
\begin{thr} \label{mth001-2np1new} Let $d^{(j)}$ be as in
 (\ref{leipg0})-(\ref{leipgj}) and
  $\rb_j$, $\tb_j$, $\lb_j$,  $\mb_j$
 be defined by (\ref{ll-2np1})-(\ref{lllj}), respectively. The identity
 \begin{align}
 d^{(2j+1)}(z)
  &=
\left(       \begin{array}{cc}
             I_q & \rb_j \\
             0_q & I_q \\
           \end{array}
         \right)
         \left(
           \begin{array}{cc}
             I_q & 0_q \\
             -(z-a)\mb_j & I_q \\
           \end{array}
         \right)
         \left(
           \begin{array}{cc}
             I_q & -\rb_j \\
             0_q & I_q \\
           \end{array}
         \right),
         \label{main1-2np1}
 \\
d^{(2j+2)}(z)
  &=
\left(
           \begin{array}{cc}
             I_q & 0_q \\
             -\tb_j & I_q \\
           \end{array}
         \right)
         \left(
           \begin{array}{cc}
             I_q & (z-a)\lb_j\\
             0_q & I_q \\
           \end{array}
         \right)
         \left(
           \begin{array}{cc}
             I_q & 0_q \\
             \tb_j & I_q \\
           \end{array}
         \right)
         \label{main1-2n}
 \end{align}
 then holds for $0\leq j \leq n$.
\end{thr}
 \begin{proof} We prove (\ref{main1-2np1}).
  For $j=0$ the proof can be checked  by a direct calculation. %
  Let  $1\leq j\leq n$.  Denote
$$
    d^{(2j+1)}:=\left(
          \begin{array}{cc}
            d_j^{11} & d_j^{12} \\
            d_j^{21} & d_j^{22} \\
          \end{array}
        \right). 
$$ 
   The relation (\ref{main1-2np1}) is equivalent to the following four
   equalities:
   \begin{align}
&d_j^{11}-I_q+(z-a)\rb_j \mb_j=0, \label{f11lei2np1}\\
&d_j^{12}-(z-a)\rb_j \mb_j \rb_j=0, \label{f12lei2np1}\\
&d_j^{21}+(z-a)\mb_j=0, \label{f21lei2np1}\\
&d_j^{22}-I_q-(z-a)\mb_j \rb_j=0. \label{f22lei2np1}
   \end{align}
Let us now prove (\ref{f11lei2np1}). By the $(1,1)$ element of
$d^{(2j+1)}$ , (\ref{mm01-2np1}) and (\ref{tQP2j}), we have
\begin{align*}
&d_j^{11}-I_q+(z-a)\rb_j \mb_j\\
&=-(z-a)\Theta_{1,j}^*(a)\widehat K_{1,j}^{-1}
\Gamma_{1,j}(a)+(z-a)\Theta_{1,j}^*(a)\Gamma_{1,j}^{*^{-1}}(a)
\Gamma_{1,j}^*(a)\widehat K_{1,j}^{-1}\Gamma_{1,j}(a)\\
&=0
\end{align*}
The equalities (\ref{f12lei2np1}) and (\ref{f22lei2np1}) are proved
in a similar way. Note that (\ref{f21lei2np1}) is verified by
definition. To prove (\ref{main1-2n}) one uses (\ref{leipgj}), the
first equality of (\ref{mm01-2np1}) and the second equality of
(\ref{tQP2j}). Thus theorem \ref{mth001-2np1new} is proved.
\end{proof}
Let $n\in \N_0$, and let $A_0,\ldots, A_n$ be complex $q\times q$
matrices. Then let
\begin{align*}
\overrightarrow{\prod_{j=0}^{n}} A_j=A_0A_1\cdots A_{n-1} A_n \quad
\mbox{and} \quad \overleftarrow{\prod_{j=0}^{n}}
A_j=A_nA_{n-1}\cdots A_1 A_0
\end{align*}
denote the right and left product of the matrices $A_0, A_1, \ldots,
A_n$, respectively.
 \vskip3mm

The following corollary readily yields by employing (\ref{mm100a}),
(\ref{mm200a}), Theorem \ref{mth001-2np1new} and
Corollary \ref{cor-2np1}.                     %
\begin{sleds}\label{utildes}
Let $\widehat U_2^{(2n)}$ and $\widetilde U_2^{(2n+1)}$ be as in
(\ref{RM1000}) and (\ref{leipRMe22}), respectively.
 Thus, the equalities
 \begin{equation}
  \widehat U_2^{(2n)}=\overrightarrow{\prod_{k=0}^{n-1}}
  \left[\left(
  \begin{array}{cc}
  I_q & (z-a)\lb_{k-1} \\
        0_q & I_q \\
         \end{array}
         \right)\left(
                       \begin{array}{cc}
                         I_q & 0_q \\
                         -\mb_k & I_q \\
                       \end{array}
                     \right)
                    \right] \label{0009}
\end{equation}
  and
\begin{equation}
  \widetilde U_2^{(2n+1)}=
  \overrightarrow{\prod_{k=0}^{n}}
  \left[\left(\begin{array}{cc}
        I_q &  \lb_{k-1} \\
        0_q & I_q \\
        \end{array}  \right)
  \left( \begin{array}{cc}
           I_q & 0_q \\
           -(z-a)\mb_k & I_q \\
                       \end{array}
                     \right)
                     \right]
   \label{00010}
\end{equation}
   are valid.
%
\end{sleds}
 Now we derive a new representation of the RM of the THMM problem via
 DSM parameters
 in both cases for odd and even number of moments.
 We also reproduce an analogue representation given in
 \cite[Corollary 3]{abdonDS1}.
 To this end let us recall the DSM parameters first introduced in
 \cite{abODD}:
\begin{align}
 \Mb_0(a):=&s_0^{-1},
 \quad \Lb_0(a):=\widetilde u^*_{2,0}K_{2,0}^{-1}\widetilde
 u_{2,0}, \label{MM0a}\\
 \Mb_j(a):=&v_{j}^*R_{j}^*(a)H_{j}^{-1}R_{j}(a)v_{j}-
 v_{j-1}^*R_{j-1}^*(a)H_{1,j-1}^{-1}R_{j-1}(a)v_{j-1},\label{MMja}\\
\Lb_j(a):=&\widetilde u^*_{2,j}R_j^*(a)K_{2,j}^{-1}R_j(a)\widetilde
u_{2,j} -\widetilde
u^*_{2,j-1}R_{j-1}^*(a)K_{2,j-1}^{-1}R_{j-1}(a)\widetilde u_{2,j-1}.
\label{LLja}
\end{align}
\begin{thr}\label{mainT1}
  Let $P_{k,j}$, $Q_{k,j}$, $\Gamma_{k,j}$ and $\Theta_{k,j}$ be as in
Definitions \ref{de002} and \ref{def001}. Furthermore, let
 the RM $U^{(m)}$ of the THMM problem be as in (\ref{uujd}),
 (\ref{uujdeven}).\\
a)  
 Let $\lb_k$, $\mb_k$, be as in Definition  (\ref{mmm01})--(\ref{lllj}).
 Thus, the following representation of the resolvent
  matrix in the case of odd numbers
 of moments holds
 \begin{align}
  U^{(2n)}&(z,a,b)=\left(
                      \begin{array}{cc}
                        \frac{1}{(b-z)(z-a)}I_q & 0_q \\
                        0_q & I_q \\
                      \end{array}
                    \right)
\overrightarrow{\prod_{k=0}^{n-1}}
  \left[\left(
  \begin{array}{cc}
  I_q & (z-a)\lb_{k-1} \\
        0_q & I_q \\
         \end{array}
         \right)\left(
                       \begin{array}{cc}
                         I_q & 0_q \\
                         -\mb_k & I_q \\
                       \end{array}
                     \right)
                    \right]
                    \nonumber\\
                   &
                   \cdot\left(
  \begin{array}{cc}
  I_q & (z-a)\lb_{n-1} \\
        0_q & I_q \\
         \end{array}
         \right)\left(
                     \begin{array}{cc}
                       I_q & 0_q \\
                       Q_{2,n-1}^{-1}(a)P_{2,n-1}(a)
                       +\frac{1}{b-a}\Theta_{2,n}^{-1}(a)\Gamma_{2,n}(a)
                       & I_q \\
                     \end{array}
                   \right)\nonumber\\
 &\cdot
 \left(
                      \begin{array}{cc}
                        (b-a)(z-a)I_q & 0_q \\
                        0_q & \frac{b-z}{b-a}I_q \\
                      \end{array}
                    \right),\label{mf-2n}
                    \\
 U^{(2n+1)}&(z,a,b)=\left(
                      \begin{array}{cc}
                        \frac{1}{b-z}I_q & 0_q \\
                        0_q & I_q \\
                      \end{array}
                    \right)
\overrightarrow{\prod_{k=0}^{n}}
  \left[\left(\begin{array}{cc}
        I_q &  \lb_{k-1} \\
        0_q & I_q \\
        \end{array}  \right)
  \left( \begin{array}{cc}
           I_q & 0_q \\
           -(z-a)\mb_k & I_q \\
                       \end{array}
                     \right)
                     \right]\nonumber\\
                     &\cdot\left(
                       \begin{array}{cc}
                         I_q & -\Gamma_{1,n}^{-1}(a)
                         \Theta_{1,n}(a)-(b-a)P_{1,n+1}^{-1}(a)Q_{1,n+1}(a) \\
                         0_q & I_q
                       \end{array}
                     \right)
 \left(
                      \begin{array}{cc}
                        (b-z)I_q & 0_q \\
                        0_q & I_q \\
                      \end{array}
                    \right).\label{mf-2np1}
      \end{align}
b)  Moreover, let $\Mb_k$, $\Lb_k$ be as in (\ref{MM0a})--
(\ref{LLja}). Thus the following representations  hold:
 \begin{align}
  U^{(2n)}&(z,a,b)=\overrightarrow{\prod_{k=0}^{n-1}}
  \left[\left(
                       \begin{array}{cc}
                         I_q & 0_q \\
                         -(z-a) \Mb_k & I_q \\
                       \end{array}
                     \right)\left(
                              \begin{array}{cc}
                                I_q &  \Lb_k \\
                                0_q & I_q \\
                              \end{array}
                            \right)\right]
                             \left(
                       \begin{array}{cc}
                         I_q & 0_q \\
                         -(z-a) \Mb_n & I_q \\
                       \end{array}
                     \right)\nonumber\\
                           &\cdot\
                     \left(
                       \begin{array}{cc}
                         I_q & Q_{1,n}^*(a)P_{1,n}^{*^{-1}}(a)+\frac{1}{b-a}
                         \Theta_{1,n}^*(a)\Gamma_{1,n}^{*^{-1}}(a) \\
                         0_q & I_q \\
                       \end{array}
                     \right)
      \label{mf1a}
      \end{align}
      and
      \begin{align}
 U^{(2n+1)}&(z,a,b)=\left(
                      \begin{array}{cc}
                        \frac{1}{z-a}I_q & 0_q \\
                        0_q & I_q \\
                      \end{array}
                    \right)
 \overrightarrow{\prod_{k=0}^{n}}
  \left[\left(
                       \begin{array}{cc}
                         I_q & 0_q \\
                         -\Mb_k & I_q \\
                       \end{array}
                     \right)
                     \left(
                              \begin{array}{cc}
                                I_q &  (z-a)\Lb_k \\
                                0_q & I_q \\
                              \end{array}
                            \right)\right]\nonumber\\
                           &
                    \cdot
\left(
                       \begin{array}{cc}
                         I_q & 0_q \\
                         -\Gamma_{2,n}^*(a)\Theta_{2,n}^{*^{-1}}(a)-(b-a)
                         P_{2,n}^*(a)Q_{2,n}^{*^{-1}}(a) & I_q \\
                       \end{array}
                     \right) 
                     \left(
                      \begin{array}{cc}
                        (z-a)I_q & 0_q \\
                        0_q & I_q \\
                      \end{array}
                    \right).\label{mf2a}
\end{align}
\end{thr}
\begin{proof} We prove part a).
 Equality (\ref{mf-2n}) is proved using (\ref{eq29aa}),
  (\ref{eq290z}), (\ref{lei6.15}) and  (\ref{0009}).
   In a similar manner one proves
   equality (\ref{mf-2np1}) by using (\ref{eq29aa0}),
  (\ref{00010}) and  (\ref{leipB12nm1}). 
  Part b) is proved in \cite[Corollary 3]{abdonDS1}.
\end{proof}


\section{Extremal solutions of the THMM problem via continued fractions
in terms of DSM parameters}
 As a consequence of the multiplicative representation of $U^{(2n)}$
 and $U^{(2n+1)}$ as in (\ref{mf-2n}) and (\ref{mf-2np1}), we attain
 a representation of the extremal solutions of the THMM problem
 through continued fractions in terms of DSM parameters.

 Set $\frac{A}{B}:=AB^{-1}$ for $A,\, B\in {\mathbb C}^{q\times q}$ with $B$
 invertible.
 \begin{defn}
  Let be $P_{k,n}$, $Q_{k,n}$ as in Definition
  \ref{de002}, and let $\Theta_{k,n}$, $\Gamma_{k,n}$ be as in Definition
 \ref{def001}.
%
 The following $q\times q$ matrix valued functions defined for all $z\in \C\setminus
[a,b]$:
  \begin{align}
  s_{\mathcal K}^{(2n)}(z):=& \frac{\Theta_{2,n}^*(\bar z)}{(z-a) \Gamma_{2,n}^*(\bar z)},
  \quad
    s_{\mathcal F}^{(2n)}(z):=\frac{\Theta_{1,n}^*(\bar z)}{(b-z)\Gamma_{1,n}^*(\bar
    z)},
    \label{extrem2n}
    \\
  s_{\mathcal K}^{(2n+1)}(z) :=& -\frac{Q_{2,n}^*(\bar z)}{(z-a)(b-z) P_{2,n}^*(\bar z)},
  \quad
    s_{\mathcal F}^{(2n+1)}(z)
    :=-\frac{Q_{1,n+1}^*(\bar z)}{P_{1,n+1}^*(\bar z)}
    \label{extrem2nmas1}
    \end{align}
are called the even (resp. odd)
  Krein and Friedrichs  extremal solutions of the THMM problem.
 \end{defn}
Note that  a justification of the adjective ``extremal" has been not
yet given. Such a justification will be
 presented in a forthcoming work. In \cite[Definition 4]{dyu0} the extremal solutions
 of the Stieltjes matrix moment problem  are the terminal matrices of
 a matrix interval.

 \vskip1mm
  To obtain a continued fraction representation of the extremal
  functions, we write a slightly different multiplicative
 representation of the RM $U^{(m)}$; see (\ref{mf-2n})--(\ref{mf2a}).
  Let us introduce some additional  notation:
\begin{align}
 \LL_k^{(2n)}:=&\left(
  \begin{array}{cc}
  0_q&I_q\\
  I_q & (z-a)\lb_{k}
         \end{array}
         \right), \quad
 \MM_k^{(2n)}:=\left(
                       \begin{array}{cc}
                         0_q & I_q \\
                         I_q&-\mb_k
                       \end{array}
                     \right),
                     \label{mmll000A}\\
 \LL_k^{(2n+1)}:=&\left(\begin{array}{cc}
        0_q& I_q\\
        I_q &  \lb_{k}
        \end{array}  \right),
         \quad
 \MM_k^{(2n+1)}:=
         \left( \begin{array}{cc}
           0_q & I_q \\
           I_q&-(z-a)\mb_k
                       \end{array}
                     \right),
                                \label{mmll000}\\
\LLC_k^{(2n)}:=&\left(
  \begin{array}{cc}
  0_q&I_q\\
  I_q & \Lb_{k}
         \end{array}
         \right),
          \quad
  \MMC_k^{(2n)}:=\left(
                       \begin{array}{cc}
                         0_q & I_q \\
                         I_q&-(z-a)\Mb_k
                       \end{array}
                     \right),
                     \label{mm300}\\
 \LLC_k^{(2n+1)}:=&\left(\begin{array}{cc}
        0_q& I_q\\
        I_q &  (z-a)\Lb_{k}
        \end{array}  \right),
         \quad
  \MMC_k^{(2n+1)}:=
         \left( \begin{array}{cc}
           0_q & I_q \\
           I_q&-\Mb_k
                       \end{array}
                     \right).
           \label{mm400}
 \end{align}
 and
  \begin{align}
   \mathscr{D}_1:=&\left(
                      \begin{array}{cc}
                        0_q&\frac{1}{(b-z)(z-a)}I_q \\
                        I_q & 0_q \\
                      \end{array}
                    \right),\quad
   \mathscr{D}_2:=\left(
                      \begin{array}{cc}
                        0_q&\frac{b-z}{b-a}I_q\\
                        (b-a)(z-a)I_q &0_q
                      \end{array}
                    \right),\label{D111}\\
\mathscr{D}_3:=&\left(
                      \begin{array}{cc}
                      0_q &\frac{1}{b-z}I_q\\
                        I_q & 0_q
                      \end{array}
                    \right),
                    \quad
{\bf D}_4:=\left(
                      \begin{array}{cc}
                        (b-z)I_q&0_q\\
                        0_q & I_q
                      \end{array}
                    \right), \nonumber 
                    \\
{\bf D}_5:=&\left(
                      \begin{array}{cc}
                      \frac{1}{z-a}I_q&0_q\\
                        0_q &I_q
                      \end{array}
                    \right),
                    \quad
\mathscr{D}_6:=\left(
                      \begin{array}{cc}
                        0_q&I_q\\
                       (z-a) I_q & 0_q
                      \end{array}
                    \right),\nonumber
                    \\
\mathscr{B}^{(2n)}_1:=&\left(
                     \begin{array}{cc}
                       0_q & I_q \\
                       I_q& Q_{1,n}^*(a)P_{1,n}^{*^{-1}}(a)+\frac{1}{b-a}
                         \Theta_{1,n}^*(a)\Gamma_{1,n}^{*^{-1}}(a)
                     \end{array}
                   \right), \nonumber
                    \\
\mathscr{B}^{(2n+1)}_1:=&\left(
                       \begin{array}{cc}
                        0_q& I_q\\
                        I_q&-\Gamma_{2,n}^*(a)\Theta_{2,n}^{*^{-1}}(a)-(b-a)
                         P_{2,n}^*(a)Q_{2,n}^{*^{-1}}(a)
                       \end{array}
                     \right), \nonumber
                     \\
\mathscr{B}^{(2n)}_2:=&\left(
                     \begin{array}{cc}
                       0_q & I_q \\
                       I_q&Q_{2,n-1}^{-1}(a)P_{2,n-1}(a)
                       +\frac{1}{b-a}\Theta_{2,n}^{-1}(a)\Gamma_{2,n}(a)
                     \end{array}
                   \right),
                   \label{b2n01}
                    \\
\mathscr{B}^{(2n+1)}_2:=&\left(
                       \begin{array}{cc}
                        0_q& I_q\\
                        I_q& -\Gamma_{1,n}^{-1}(a)
                         \Theta_{1,n}(a)-(b-a)P_{1,n+1}^{-1}(a)Q_{1,n+1}(a)
                       \end{array}
                     \right).\nonumber
  \end{align}
%
\begin{lem}\label{lem2016}
 Let the RM $U^{(2n)}$ and $U^{(2n+1)}$ be defined as in (\ref{uujd}) and
  (\ref{uujdeven}).
 The identities (\ref{0003}), (\ref{0004})
\begin{equation}
U^{(2n)}=\MMC_0^{(2n)}\LLC_0^{(2n)}\ldots \LLC_{n-1}^{(2n)}
\MMC_{n}^{(2n)}\mathscr{B}^{(2n)}_1  \label{0005}
 \end{equation}
and
\begin{equation}
 U^{(2n+1)}={\bf D}_5 \MMC_0^{(2n+1)}\LLC_0^{(2n+1)}\ldots
 \MMC_n^{(2n+1)}\LLC_{n}^{(2n+1)}\mathscr{B}^{(2n+1)}_1 \frak{D}_6
 \label{0006}
\end{equation}
hold.
\end{lem}
\begin{proof} We prove (\ref{0003}). By using (\ref{D111}), (\ref{mmll000A}) and (\ref{b2n01})
 clearly the following equalities are valid:
\begin{align*}
\mathscr{D}_1 \LL_{-1}^{(2n)}=&
 \left(
   \begin{array}{cc}
     \frac{1}{(b-z)(z-a)}I_q & 0_q \\
     0_q & I_q
   \end{array}
 \right)\left(
          \begin{array}{cc}
            I_q & (z-a)\lb_{-1} \\
            0_q & I_q
          \end{array}
        \right)
        \\
\MM_k^{(2n)} \LL_k^{(2n)}=&
        \left(
          \begin{array}{cc}
            I_q & 0_q \\
            -\mb_k & I_q
          \end{array}
        \right)
        \left(
          \begin{array}{cc}
            I_q & (z-a)\lb_k \\
            0_q & I_q
          \end{array}
        \right),
        \\
\mathscr{B}^{(2n)}_2\mathscr{D}_2=&
 \left(
                     \begin{array}{cc}
                       I_q & 0_q \\
                       Q_{2,n-1}^{-1}(a)P_{2,n-1}(a)
                       +\frac{1}{b-a}\Theta_{2,n}^{-1}(a)\Gamma_{2,n}(a)
                       & I_q \\
                     \end{array}
                   \right)\\
                   &\cdot
                   \left(
                      \begin{array}{cc}
                        (b-a)(z-a)I_q & 0_q \\
                        0_q & \frac{b-z}{b-a}I_q \\
                      \end{array}
                    \right).
 \end{align*}
 The latter equalities along with (\ref{mf-2n}) imply (\ref{0003}).
 In a similar manner ones proves (\ref{0004}), (\ref{0005}) and
 (\ref{0006}).
\end{proof}


 Following \cite[Page 11]{simon}, we consider
 $\C^{q\times q}\bigotimes \C^{q\times q}$
 as a right module over $\C^{q\times q}$

For $\mathfrak U:=\left(
                    \begin{array}{cc}
                      A & B \\
                      C & D
                    \end{array}
                  \right)$ define the transformation
\begin{equation}
{\bf \mathscr{T}}_{\mathfrak U}\left(
                   \begin{array}{c}
                   X \\   Y \\
                  \end{array} \right)
                  =\left(
                            \begin{array}{c}
                              (AX+BY)(CX+DY)^{-1} \\
                              I_q
                            \end{array}
                          \right)\label{XX1}
\end{equation} if $CX+DY$ is invertible.
\begin{thr}\label{mainT2} Let $Q_{k,j}$, $P_{k,j}$, $\Gamma_{k,j}$ and
$\Theta_{k,j}$ for $k=1,2$ be as in Definition \ref{de002} and
Definition \ref{def001}.
 The extremal solutions (\ref{extrem2n}) and (\ref{extrem2nmas1})
 of the THMM problem can be represented via
 finite matrix continued fractions:
 \begin{align}
 &\frac{\Theta_{1,n}^*(\bar z)}{(b-z)\Gamma_{1,n}^*(\bar
z)}\nonumber\\
&=\frac{s_0}{b-z} +\cfrac{I_q}{-(z-a)(b-z)\mb_0+\cfrac{I_q}
{\frac{1}{b-z}\lb_0+\cfrac{I_q}{\ddots \,
 -(z-a)(b-z)\mb_{n-1}+\cfrac{I_q}{\frac{1}{b-z}\lb_{n-1}
} }}}, \label{cc23zz}\\
 &\frac{\Theta_{2,n}^*(\bar z)}{(z-a)\Gamma_{1,n}^*(\bar z)}=
 -\cfrac{I_q}{-(z-a)\Mb_0+\cfrac{I_q}{\Lb_0
 +\cfrac{I_q}{\ddots 
 +L_{n-1}+
 \cfrac{I_q}{
-(z-a)\Mb_n
}}}}, \label{cc13zz}
 \end{align}
\begin{align}
 &-\frac{Q_{1,n+1}^*(\bar z)}{P_{1,n+1}^*(\bar z)}=
 \cfrac{I_q}{-(z-a)\Mb_0+\cfrac{I_q}{\Lb_0
 +\cfrac{I_q}{\ddots 
 +L_{n-1}+
 \cfrac{I_q}{
-(z-a)\Mb_{n}+\Lb_{n}^{-1}}}}}, \label{cc13z}
\\
&\frac{Q_{2,n}^*(\bar z)}{(z-a)(b-z)P_{2,n}^*(\bar
z)}\nonumber\\
&=\frac{s_0}{b-z}
+\cfrac{I_q}{-(z-a)(b-z)\mb_0+\cfrac{I_q}{\frac{1}{b-z}\lb_0+\cfrac{I_q}{\ddots
\,
 \frac{1}{b-z}\lb_{n-1}+\cfrac{I_q}{-(z-a)(b-z)\mb_n
} }}}. \label{cc23z}
 \end{align}
\end{thr}
\begin{proof} We prove (\ref{cc23zz}).
 We apply the transformation (\ref{XX1}) at
  $\begin{posmallmatrix}X\\Y\end{posmallmatrix}
=\begin{posmallmatrix}0_q\\I_q\end{posmallmatrix}$
 to both sides of (\ref{mf-2n}), more precisely,
 on the left side
${\mathfrak U}=U^{(2n)}$, however, on the right side
 ${\mathfrak U}$ is equal to $\mathscr{D}_1 \LL_{-1}^{(2n)}\MM_0^{(2n)}\ldots
\MM_{n-1}^{(2n)} \LL_{n-1}^{(2n)}\mathscr{B}^{(2n)}_2
\mathscr{D}_2$.
 Taking into account the composition property of the
 M\"{o}bius transformation
 $${\bf \mathscr{T}}_{U^{(2n)}}={\bf \mathscr{T}}_{\mathscr{D}_1}\circ
 {\bf \mathscr{T}}_{\MM_0^{(2n)}}\circ{\bf \mathscr{T}}_{\LL_0^{(2n)}}
 \circ\ldots
 \circ{\bf \mathscr{T}}_{\MM_{n-1}^{(2n)}}
 {\bf \mathscr{T}}_{\mathscr{B}^{(2n)}_2}\circ
 {\bf \mathscr{T}}_{\mathscr{D}_2},
$$
 we attain the following equality:
 \begin{align*}
 &\left(\begin{array}{c}
     \frac{1}{b-z} \Theta_{1,n}^*(\bar z)\Gamma_{1,n}^{-1^*}(\bar
z)\\ I_q\end{array} \right)
=\\
 &
 \left(
   \begin{array}{c}
    \frac{1}{b-z}\lb_{-1}+ (-(b-z)(z-a)\mb_0+(\frac{1}{b-z}\lb_0
     +\ldots
     +(-(b-z)(z-a)\mb_n+\frac{1}{b-z}\lb_n^{-1})^{-1}\ldots)^{-1}\\
     I_q \\
   \end{array}
 \right).
\end{align*}
 Equality (\ref{cc23zz}) is proved.
 Equality (\ref{cc13zz})
  can be verified applying the M\"{o}bius transformation $\bf
  \mathscr{T}$ to both sides of (\ref{0005}) at
  $\begin{posmallmatrix}X\\Y\end{posmallmatrix}
=\begin{posmallmatrix}I_q\\0_q\end{posmallmatrix}$.
  In a similar way, to prove (\ref{cc13z}) one uses (\ref{0006})
at
  $\begin{posmallmatrix}X\\Y\end{posmallmatrix}
=\begin{posmallmatrix}0_q\\I_q\end{posmallmatrix}$.
 To verify (\ref{cc23z}) we employ the transformation $\mathscr{T}$
 to equality (\ref{0004}), more precisely,
  from one side for ${\mathfrak U}=
  U^{(2n+1)}{\bf D}_4^{-1}$
   on the other side ${\mathfrak U}$ is equal to
 $\MMC_0^{(2n+1)}\LLC_0^{(2n+1)}\ldots
 \MMC_n^{(2n+1)}\LLC_{n}^{(2n+1)}\mathscr{B}^{(2n+1)}_1 \frak{D}_6
 $ at
  $\begin{posmallmatrix}X\\Y\end{posmallmatrix}
=\begin{posmallmatrix}I_q\\0_q\end{posmallmatrix}$.
\end{proof}

\section{Relations between the OMP, DSM parameters and Schur complements}
In this section, we obtain explicit relations between the OMP on the
interval $[a,b]$ (as in Definitions \ref{de002} and \ref{def001}),
 the  DSM parameters (\ref{mmm01})-(\ref{lllj}) and the Schur
complements $\widehat H_{2,j}$ and $\widehat K_{1,j}$; see
 (\ref{lkk2}) and (\ref{lkk3}).
\begin{prop}
\label{propa1new} Let the polynomials $P_{2,j}$, $Q_{2,j}$,
$\Gamma_{1,j}$ and $\Theta_{1,j}$  be defined as in Definitions
\ref{de002} and \ref{def001}. Let the DSM parameters $\mb_j$ and
$\lb_j$ be as in (\ref{mmm01})-(\ref{lllj}).
 The following identities then hold:
\begin{align}
 Q_{2,j}(a)=&(-1)^{j}
 \mb_{0}^{-1}\lb_{0}^{-1}\mb_{1}^{-1}
 \lb_{1}^{-1}\cdots \lb_{j-1}^{-1}\mb_{j}^{-1},
  \ \ 0\leq j\leq n-1,\label{f022j}
  \\
 \Gamma_{1,j}(a)=&(-1)^{j}
 \mb_{0}^{-1}\lb_{0}^{-1}\mb_{1}^{-1}
 \lb_{1}^{-1}\cdots \lb_{j-1}^{-1}, \ \
 1\leq j\leq n,\label{f0021j}
 \\
  P_{2,j}(a)=&(-1)^{j}
  \mb_0^{-1}\lb_0^{-1} \mb_1^{-1} \lb_1^{-1}\cdots
  \lb_{j-1}^{-1}\mb_j^{-1} 
  (\mb_0+\mb_1+\ldots+\mb_{j-1}+\mb_j), \ \ 1\leq j\leq n-1
  \label{f021j}
\end{align}
and
\begin{align}
 \Theta_{1,j}(a)=
 &(-1)^{j} \mb_{0}^{-1}\lb_{0}^{-1}\mb_{1}^{-1}\lb_{1}^{-1}
 \cdots \mb_{j-1}^{-1}\lb_{j-1}^{-1}
(s_0+\lb_0+\ldots +\lb_{j-1}),\ \ 1\leq j\leq n. \label{f012j}
\end{align}
\end{prop}
\begin{proof}
 By rewriting the equality (\ref{eq29aa}) in the form
 \begin{align*}
 \left(
   \begin{array}{cc}
     (z-a)(b-z)I_q & 0_q \\
     0_q & I_q
   \end{array}
 \right)U^{(2j)}=\left(
                  \begin{array}{cc}
                    \frac{1}{(z-a)(b-z)}I_q & 0_q \\
                    0_q & \frac{b-a}{b-z}I_q
                  \end{array}
                \right)\widehat U_2^{(2j-2)}\left(
                         \begin{array}{cc}
                           I_q & 0_q \\
                           N_{2,j} & I_q
                         \end{array}
                       \right),
 \end{align*}
 by expanding the right hand side of this equality
  in powers of $(z-a)$ and by employing (\ref{uujd}), (\ref{0009}),
   (\ref{lei6.15}),
    we have
\begin{align}
 \frac{b-z}{b-a}\Theta_{2,j}^*(\bar z)\Theta_{2,j}^{*^{-1}}(a)
 =&I_q+\ldots+
 (-1)^{j}(z-a)^{j+1} s_0\mb_0\lb_0\cdots \mb_j\lb_j\nonumber\\
 &\cdot
 \left(P_{2,j-1}^{*}(a)Q_{2,j-1}^{*^{-1}}(a)+
 \frac{1}{b-a}\Gamma_{2,j}^{*}(a)\Theta_{2,j}^{*^{-1}}(a)
\right),
 \label{neu11p}\\
 \Theta_{1,j}^*(\bar z)\Gamma_{1,j}^{*^{-1}}(a)
 =
 &s_0+\lb_0+\ldots+\lb_{j-1}+\cdots+
  (-1)^j(z-a)^js_0\mb_0\lb_0\cdots \lb_{j-1},
\label{neu12p}
\end{align}
\begin{align} \frac{1}{b-a} \Gamma_{2,j}^*(\bar
 z)\Theta_{2,j}^{*^{-1}}(a)=&-(\mb_0+\ldots+\mb_{j-1})
 \nonumber\\
 &+\left(P_{2,j-1}^{*}(a)Q_{2,j-1}^{*^{-1}}(a)+
 \frac{1}{b-a}\Gamma_{2,j}^{*}(a)\Theta_{2,j}^{*^{-1}}(a)
\right)\nonumber\\
&+\ldots +(-1)^{j}(z-a)^j \mb_0\lb_0\cdots \mb_{j-1}
\lb_{j-1}\nonumber\\
&\cdot\left(P_{2,j-1}^{*}(a)Q_{2,j-1}^{*^{-1}}(a)+
 \frac{1}{b-a}\Gamma_{2,j}^{*}(a)\Theta_{2,j}^{*^{-1}}(a)
\right), \label{neu21p}
\\
 \Gamma_{1,j}^*(\bar z)\Gamma_{1,j}^{*^{-1}}(a)=&I_q+\ldots
 +(-1)^{j}(z-a)^{j} \mb_0\lb_0 \mb_1\lb_1\cdots
\mb_{j-1}\lb_{j-1}. 
 \label{neu22p}
\end{align}
Similarly by rewriting (\ref{eq29aa0}) in the form
$$
\left(
  \begin{array}{cc}
    (b-z)I_q & 0_q \\
    0_q & I_q
  \end{array}
\right)U^{(2j+1)}\left(
                   \begin{array}{cc}
                     \frac{1}{b-z}I_q & 0_q \\
                     0_q & I_q
                   \end{array}
                 \right)=\widetilde U_2^{(2j+1)} A_2^{(2j+1)}
$$
and by using (\ref{uujdeven}),   (\ref{00010}), 
(\ref{leipB12nm1}) and (\ref{uRj02}), we have
\begin{align}
 Q_{2,j}^*(\bar z)Q_{2,j}^{*^{-1}}(a)=&
 I_q+\ldots+(-1)^{j+1}(z-a)^{j+1}s_0\mb_0\lb_0\mb_1\lb_1\cdots
\lb_{j-1}\mb_j,
\label{neu11i}
\\
(z-b)Q_{1,j+1}^*(\bar z)P_{1,j+1}^{*^{-1}}(a)=
 &s_0+\lb_0+\ldots+\lb_{j-1}-\Theta_{1,j}^*(a)\Gamma_{1,j}^{*^{-1}}(a)\nonumber\\
 &-(b-a)Q_{1,j+1}^*(a)P_{1,j+1}^{*^{-1}}(a)
 \nonumber\\
 &+\ldots+(-1)^{j+1}(z-a)^j s_0\mb_0 \lb_0\mb_1\lb_1\cdots
\lb_{j-1}\mb_j\nonumber\\
&\cdot (-\Theta_{1,j}^*(a)\Gamma_{1,j}^{*^{-1}}(a)
 -(b-a)Q_{1,j+1}^*(a)P_{1,j+1}^{*^{-1}}(a))\nonumber
 \\
 &+\ldots+ (-1)^{j+1}(z-j)^{j+1}s_0\mb_0 \lb_0\mb_1\lb_1\cdots
\lb_{j-1}\mb_j\nonumber\\
&\cdot(-\Theta_{1,j}^*(a)\Gamma_{1,j}^{*^{-1}}(a)
 -(b-a)Q_{1,j+1}^*(a)P_{1,j+1}^{*^{-1}}(a)),
\label{neu12i}
\\
 P_{2,j}^*(\bar z)Q_{2,j}^{*^{-1}}(a)=&\mb_0+\mb_1+\ldots+\mb_{j-1}
 \nonumber\\
 &+\ldots+ (-1)^j(z-a)^j\mb_0\lb_0\mb_1\lb_1\cdots\lb_{j-1}\mb_j,
\label{neu21i}
\\
 P_{1,j+1}^*(\bar z)P_{1,j+1}^{*^{-1}}(a)=&I_q+\ldots
 +(-1)^{j+1}(z-a)^{j+1} \mb_0\lb_0\mb_1\lb_1\cdots
\lb_{j-1}\mb_j \nonumber\\
&\cdot (-\Theta_{1,j}^*(a)\Gamma_{1,j}^{*^{-1}}(a)
 -(b-a)Q_{1,j+1}^*(a)P_{1,j+1}^{*^{-1}}(a)).
 \label{neu22i}
\end{align}
Equalities (\ref{f022j}), (\ref{f0021j}), (\ref{f021j})
 and (\ref{f012j}) follow
from (\ref{neu12p}), (\ref{neu22p}) and (\ref{neu12i}). Employ the
fact that $P_{2,j}$,
$\Gamma_{1,j}$ are monic polynomials. 
\end{proof}
Note that from (\ref{neu22i}) and (\ref{neu21p}), one obtains the
following identities:
 \begin{align}
 (b-a)Q_{1,j+1}(a)-Q_{2,j}(a)+P_{1,j+1}(a)\Gamma_{1,j}^{-1}(a)\Theta_{1,j}(a)=&0,
 \label{nuw1}\\
 \Gamma_{1,j}(a)-
 \Gamma_{2,j}(a)-(b-a)\Theta_{2,j}(a)Q_{2,j-1}^{-1}(a)
 P_{2,j-1}(a)=&0.
 \label{nuw2}
  \end{align}
The following remark is verified by using identities from
Proposition \ref{propa1new} and (\ref{nuw1}), (\ref{nuw2}).
\begin{rem} \label{remaa1new}
 The following identities hold:
  \begin{align*}
 Q_{2,j}(a)=&\Gamma_{1,j}(a)\mb_j^{-1},
 \\
 \Gamma_{1,j}(a)=&-Q_{2,j-1}(a)\lb_{j-1}^{-1},\\
  P_{2,j}(a)=&\,Q_{2,j}(a)(\mb_0+\mb_1+\ldots+\mb_{j-1}),\\
  \Theta_{1,j}(a)=&\,\Gamma_{1,j}(a)(s_0+\lb_0+\ldots+\lb_{j-1}),
  \end{align*}
  and
  \begin{align*}
 Q_{2,j}(a)=&-Q_{2,j-1}(a)\lb_{j-1}^{-1}\mb_j^{-1},\\
 \Theta_{1,j}(a)=&Q_{2,j-1}\lb_{j-1}^{-1}
 (s_0+\lb_0+\ldots+\lb_{j-1}). 
  \end{align*}
\end{rem}
\begin{rem}
\label{rem4.4Anew} Let $P_{1,j}$, $Q_{2,j}$, $\Gamma_{1,j}$ and
$\Theta_{2,j}$ be  as in Definitions \ref{de002} and \ref{def001}.
Let
  $\widehat H_{1,j}$, $\widehat K_{2,j}$ be defined by 
  (\ref{lkk}),  (\ref{lkk2})
and $\mb_j,\, \lb_j$ be as in  (\ref{mmm01})-(\ref{lllj}).
  Therefore, the following identities hold:
  \begin{equation}
 \widehat H_{2,0}=(\mb_0\lb_0)^{*^{-1}}\lb_0(\mb_0\lb_0)^{-1}, \quad
  \widehat K_{1,0}
  =\mb_0^{-1}\label{ah00}
\end{equation}
and
\begin{align}
    \widehat H_{2,j}
    =&\overrightarrow{\prod_{k=0}^{j}}(\mb_k\lb_k)^{-1^*}\lb_j
\overleftarrow{\prod_{k=0}^{j-1}}(\mb_k\lb_k)^{-1},
    \label{ah1j}\\
  \widehat{K}_{1,j}
  =&\overrightarrow{\prod_{k=0}^{j}}(\mb_k\lb_k)^{-1^*}\mb_j^{-1}
\overleftarrow{\prod_{k=0}^{j}}(\mb_k\lb_k)^{-1}. \label{bh1j}
 \end{align}
Moreover,
\begin{align*}
 P_{1,j}(a)=&(-1)^j \widehat K_{2,j-1}\widehat H_{1,j-1}^{-1}\cdots \widehat K_{2,0}
 \widehat H_{1,0}^{-1}, 
 \\
\Gamma_{1,j}(a)=
 &(-1)^j \widehat H_{2,j-1}\widehat K_{1,j-1}^{-1}\cdots \widehat H_{1,1}
 \widehat H_{2,0}\widehat K_{1,0}^{-1}, 
 \\
 Q_{2,j}(a)=
 &(-1)^j \widehat K_{1,j}\widehat H_{2,j-1}^{-1}\widehat K_{1,j-1}\cdots
 \widehat H_{2,0}^{-1}\widehat K_{1,0}, 
 \\
  \Theta_{2,j}(a)=&(-1)^{j+1} \widehat H_{1,j}\widehat K_{2,j-1}^{-1}\widehat H_{1,j-1}\cdots
 \widehat K_{2,0}^{-1}\widehat H_{1,0}. 
  \end{align*}
\end{rem}

Under the same conditions as the previous remark,
 the following result is immediately implied.
\begin{rem}
\label{rem4.4BBnew}
  The following identities hold:
  \begin{align*}
 \mb_0=\widehat K_{1,0}^{-1}, \quad
  \lb_0=(\widehat H_{2,0}^{-1}\widehat K_{1,0})^{*}\widehat H_{2,0}
(\widehat H_{2,0}^{-1}\widehat K_{1,0}^{-1})
\end{align*}
and
\begin{align*}
    \mb_j=&\overrightarrow{\prod_{k=0}^{j-1}}
    (\widehat H_{2,k} \widehat K_{1,k}^{-1})^{*}
     \widehat K_{1,j}^{-1}
\overleftarrow{\prod_{k=0}^{j-1}}(\widehat H_{2,k}\widehat
K_{1,k}^{-1}),
    \\
  \lb_j=&\overrightarrow{\prod_{k=0}^{j}}
    (\widehat H_{2,k}^{-1} \widehat K_{1,k})^{*}
     \widehat H_{2,j}
\overleftarrow{\prod_{k=0}^{j}}(\widehat H_{2,k}^{-1}\widehat
K_{1,k}).
 \end{align*}
\end{rem}
\begin{prop} \label{propHnew}
 Let $a$ and $b$ be real numbers such that $a<b$.
 Furthermore, let $s_0$, $\mb_0(b)$, $(\mb_j(a,b))_{j=0}^{n-1}$ and  $(\lb_j(a,b))_{j=0}^{n}$
(resp. $s_0$, $(\mb_j(a,b))_{j=0}^{n}$ and $(\lb_j(a,b))_{j=0}^{n}$)
be sequences of positive Hermitian complex
 $q\times q$ matrices. Let  $(s_j)_{j=0}^{2n}$ (resp.
 $(s_j)_{j=0}^{2n+1}$)
 be a sequence recursively defined by
 \begin{align*}
 bs_{2j}-s_{2j+1}:=\left\{\begin{array}{ll}
                  \mb_0^{-1}, & \mbox{if}\ \ j=0,\\
                  \widetilde Y_{1,j}^*K_{1,j-1}^{-1}\widetilde Y_{1,j}+
                   \overrightarrow{\prod_{k=0}^{j}}(\mb_k\lb_k)^{-1^*}\mb_j^{-1}
\overleftarrow{\prod_{k=0}^{j}}(\mb_k\lb_k)^{-1}
                  & \mbox{if}\ \ j\geq 1
                \end{array}
 \right. 
 \end{align*}
 and
  \begin{align*}
\widehat s_{2j}:=\left\{\begin{array}{ll}
                  (\mb_0\lb_0)^{-1^*}\lb_0(\mb_0\lb_0), & \mbox{if}\ \ j=0, \\
                  Y_{2,j}^*H_{2,j-1}^{-1} Y_{2,j}+
                   \overrightarrow{\prod_{k=0}^{j}}(\mb_k\lb_k)^{-1^*}\lb_j
\overleftarrow{\prod_{k=0}^{j-1}}(\mb_k\lb_k)^{-1}
                  & \mbox{if}\ \ j\geq 1
                \end{array}
 \right. 
 \end{align*}
 for $j=0,\ldots, n$ (resp. $j=0,\ldots, n-1$).
 Thus, $K_{1,j}=\{bs_{k+j}-s_{k+j+1}\}_{k,j=0}^n$ 
  (resp. $H_{2,j}=\{\widehat s_{k+j}\}_{k,j=0}^{n-1}$) is
  a positive Hermitian matrix.
 \end{prop}
\begin{proof} By (\ref{ah00}), (\ref{ah1j}), (\ref{bh1j})
and the fact that $\mb_j$ and $\lb_j$ are positive definite, then
 \begin{align}
 &\widehat K_{1,j}>0, \quad j\in \N_0, \label{mmlhh1}\\
 &\widehat H_{2,j}>0, \quad j\in \N_0. \label{mmlhh2}
\end{align}
 Let
\begin{equation} K_{1,n}=\left(\begin{array}{cc}
              K_{1,n-1} & \widetilde Y_{1,n} \\
              \widetilde Y_{1,n}^* & bs_{2n}-s_{2n+1}
            \end{array}\right), \ \
            H_{2,n-1}=\left(\begin{array}{cc}
              H_{2,n-2} &  Y_{2,n-1} \\
               Y_{2,n-1}^* &\widehat s_{2n-2}
            \end{array}\right). \label{Hschur}
  \end{equation}
 In view of (\ref{mmlhh1}) and the first equality of (\ref{Hschur})
  (resp. (\ref{mmlhh2}) and the second equality of (\ref{Hschur})), as well as Lemma
 \ref{lemH}, we obtain that $K_{1,j}$ (resp. $H_{2,j}$) is positive
 definite.
\end{proof}

\begin{prop} \label{propHbbnew} a) Let the polynomials
 $(Q_{2,j})_{j=0}^{n-1}$ and $(\Gamma_{1,j})_{j=1}^{n}$
 be as in
 (\ref{93aa}) and (\ref{38-1}), respectively.
Thus, the following equalities hold:
 \begin{align}
  \mb_0=&Q_{2,0}^{-1}(a), \quad \mb_j=Q_{2,j}^{-1}(a)\Gamma_{1,j}(a),
  \quad 0\leq j\leq n-1, \label{lQGj}\\
  \lb_j=& -\Gamma_{1,j+1}(a)^{-1}Q_{2,j}(a), \quad 0\leq j\leq n-1.
  \label{mGQj}
 \end{align}
 b) Let the polynomials $(P_{1,j})_{j=1}^{n+1}$, and $(\Theta_{2,j})_{j=0}^{n}$
 be as in (\ref{92}) and (\ref{39-2}), respectively.
 Thus, the following equalities hold:
\begin{align}
  \Mb_0(a)=&-\Theta_{2,0}^{-1}(a), \quad \Mb_j(a)=-\Theta_{2,j}^{-1}(a)P_{1,j}(a),
  \quad 1\leq j\leq n,\label{Mj00}\\
  \Lb_j(a)=& P_{1,j+1}^{-1}(a)\Theta_{2,j}(a), \quad 0\leq j\leq n.
  \label{Lj00}
 \end{align}
%
%
 \end{prop}
 \begin{proof} Equalities (\ref{lQGj}) and (\ref{mGQj}) readily
 follow from  (\ref{f022j}) and  (\ref{f0021j}), respectively.
  Equalities (\ref{Mj00}) and (\ref{Lj00}) readily
 follow from
\begin{align*}
 P_{1,j}(a)=&(-1)^{j}
 \Mb_{0}^{-1}\Lb_{0}^{-1}\Mb_{1}^{-1}L_{1}^{-1}\cdots \Mb_{j-1}^{-1}\Lb_{j-1}^{-1},
 \ \ 1\leq j\leq n+1,
 \end{align*}
 and
 \begin{align*}
 \Theta_{2,j}(a)=&(-1)^{j+1}
\Mb_0^{-1}\Lb_0^{-1}\Mb_1^{-1}L_1^{-1} \cdots
\Lb_{j-1}^{-1}\Mb_j^{-1}, \ \
 1\leq j\leq n
\end{align*}
which were proved in \cite[Proposition 4.1]{abdonDS1}.
 \end{proof}
 In \cite{dyu0} Dyukarev created the Stieltjes parameters
  for the Stieltjes matrix moment problem
 which in our notations are given by $\Mb_0(0)$, $\Lb_0(0)$, $\Mb_j(0)$
 and $\Lb_j(0)$. See also \cite{ablaa}.

  The following remark gives the interrelation between the mentioned
  Stieltjes parameters and the DSM parameters studied in the present
  work.
 \begin{rem} \label{remlimitab} Let 
  $\Mb_j$ and $\Lb_j$ be the DSM parameters
 as in (\ref{MM0a})-(\ref{MMja}). Furthermore, let the DSM parameters be
 as in (\ref{mmm01}), (\ref{mmmj}) and (\ref{lllj}). Thus, the
 following relations are valid:
 \begin{equation}
  \Mb_j(0)=\lim_{b\to+\infty}b\, \mb_j(0,b),\qquad
    \Lb_j(0)=\lim_{b\to+\infty}b^{-1}\lb_j(0,b).
 \end{equation}
 \end{rem}
 This Remark can be verified by direct calculations.
 \section{Scalar case} \label{secscalar}
  This section  is related to the scalar version of
the DSM parameters $\mb_j$ and $\lb_j$.

 For $q=1$,  let $n\in \N$, and let
$(s_j)_{j=0}^{2n}$ (resp.
 $(s_j)_{j=0}^{2n+1}$) be a Hausdorff positive definite sequence of numbers.
  Let \begin{align}
D^{(3)}_{x,j}:=\left(
             \begin{array}{cccc}
                            s^{(3)}_0 & s^{(3)}_1 & \ldots & s^{(3)}_j \\
                            s^{(3)}_1 & s^{(3)}_2 & \ldots & s^{(3)}_{j+1} \\
                          \ldots & \ldots & \ldots & \ldots \\
                            s^{(3)}_{j-1} & s^{(3)}_j & \ldots & s^{(3)}_{2j-1} \\
                            1 & x & \ldots & x^j \\
                          \end{array}
           \right),
    \label{Dxj3}
   \end{align}
    with 
$s^{(3)}_j:=bs_j-s_{j+1}$. 
Furthermore, let  \begin{align} E^{(2)}_{x,j}:=\left(
             \begin{array}{cccc}
                            s^{(2)}_0 & s^{(2)}_1 & \ldots & s^{(2)}_j \\
                            s^{(2)}_1 & s^{(2)}_2 & \ldots & s^{(2)}_{j+1} \\
                          \ldots & \ldots & \ldots & \ldots \\
                            s^{(2)}_{j-1} & s^{(2)}_j & \ldots & s^{(2)}_{2j-1} \\
                            e_{2,0}(x) &  e_{2,1}(x) & \ldots & e_{2,j}(x)
                          \end{array}
           \right),
    \label{Exj2}
   \end{align}
 where
 \begin{align*} (e_{2,0}(x),\,  e_{2,1}(x),\ldots, e_{2,j}(x))
 :=
                                   -(u_{2,j}^*+x v_j^* s_0)R_j^*(x).&\,\,
 \end{align*}
 The matrices $D^{(2)}_{x,j}$ and  $E^{(2)}_{x,j}$
 were introduced in \cite[Subsection 2.4]{abKN}.
 The following remark is related to the scalar version of
the DSM parameters $\widetilde m_j$ and $\widetilde l_j$.
\begin{rem} \label{skarem02} Let $\widetilde m_j$
and $\widetilde l_j$ be the scalar
 version of the DSM parameters $\lb_j$ and $\mb_j$ defined by
 (\ref{mmm01})-(\ref{lllj}).
 Let $D^{(3)}_{x,j}$ and $E_{x,j}^{(2)}$
 be as in (\ref{Dxj3}) and (\ref{Exj2}), respectively. Thus, the following
  identities hold:
 \begin{align}
\widetilde l_0(a,b)=&\frac{(E_{a,0}^{(2)})^2}{H_{2,0}}, \quad
\widetilde m_j(b)=\frac{1}{bs_0-s_1},
\label{smm0}\\
  \widetilde l_j(a,b)=&\frac{(\det E^{(2)}_{a,j})^2}{\det H_{2,j}\, \det H_{2,j-1}}, \quad
 \widetilde  m_j(a,b)=
 \frac{(\det D_{a,j}^{(3)})^2}{\det K_{1,j}\, \det K_{1,j-1}},\quad j\geq 1.
  \label{smmllj}
 \end{align}
\end{rem}
\begin{proof} The first and second equality of (\ref{smm0}) are
an immediate consequence of (\ref{mm01-2np1}) 
 for $q=1$.
 The first equality of (\ref{smmllj}) is proved by employing
 (\ref{mm01-2np1}) for $q=1$, \cite[Equation (2.40)]{abKN} 
 and the relation
   \begin{align*}
\widehat H_{2,j}=\frac{\det H_{2,j}}{\det H_{2,j-1}}, \quad j\geq 1.
   \end{align*}
   The second equality of (\ref{smmllj}) can be proved in a similar
   way. Use (\ref{mm01-2np1}) for $q=1$, \cite[Equation (2.38)]{abKN} %
    and equality $\widehat
   K_{1,j}=\frac{\det K_{1,j}}{\det K_{1,j-1}}$.
 \end{proof}
 Note that scalar Stieltjes parameters of Remark \ref{skarem02}
  coincide with (\ref{scalm}) as $b$ approaches infinity and $a=0$.


\appendix
\renewcommand*{\thesection}{\Alph{section}}

\section{Appendix}
In this appendix we reproduce some results from \cite{abdonDS1},
which are used in the present work.
\begin{defn}\cite[Formula (6.2)]{abdon2} Let $(s_k)_{k=0}^{2j}$ be an odd Hausdorff positive on $[a,b]$
sequence. The $2q\times 2q$ matrix polynomial
 $$ 
\widetilde U_2^{(2j)}(z,a,b):=\left(\begin{array}{cc}
        \widetilde\alpha_2^{(2j)}(z,a,b) &
         \widetilde\beta_2^{(2j)}(z,a,b)\\
        \widetilde\gamma_2^{(2j)} (z,a,b)&
        \widetilde\delta_2^{(2j)}(z,a,b)
        \end{array}\right),\, z\in {\mathbb C}, \quad 0\leq j\leq n-1,
$$ 
            with
\begin{align*}
 \widetilde\alpha_2^{(2j)}(z,a,b):=& I_q-(z-a)u^{*}_{2,j}R_{j}^*(\bar z) H^{-1}_{2,j}R_{j}(a)v_{j},
\\
\widetilde\beta_2^{(2j)}(z,a,b):=& (z-a)u^{*}_{2,j}R_{j}^*(\bar z)
H^{-1}_{2,j}R_{j}(a)u_{2,j},\\
\widetilde\gamma_2^{(2j)}(z,a,b):=&-(z-a)v^{*}_{j}R_{j}^*(\bar z)
H^{-1}_{2,j}R_{j}(a)v_{j}
\end{align*} 
and
\begin{align*}
\widetilde\delta_2^{(2j)}(z,a,b):=&I_q+(z-a)v^{*}_{j}R_{j}^*(\bar
z)H^{-1}_{2,j}R_{j}(a)u_{2,j}
\end{align*}
 is called the second auxiliary matrix of the THMM problem in the case of an odd number
 of moments.
\end{defn}


\begin{rem} \cite[Equalities (1.30) and (1.31)]{abdonDS1} The following identities are
valid:
\begin{align}
 \Gamma_{2,j}^*(\bar z,a)\Theta_{2,j}^{*^{-1}}(a,a)=&
 -v^{*}_{j}R_{j}^*(\bar{z}) H^{-1}_{1,j}R_{j}(a)v_{j},\label{m100}
 \\
Q_{1,j+1}^*(\bar z)P_{1,j+1}^{*^{-1}}(a)=&-\widetilde
u^{*}_{2,j}R_{j}^*(\bar{z})K^{-1}_{2,j}R_{j}(a)\widetilde
u_{2,j}.\label{uRj02}
\end{align}
\end{rem}

 \vskip3mm

Finally,  let us recall the following well-known result below.
\begin{lem}\cite[Proposition 8.2.4]{ber}\label{lemH}
Let $A:=\left(\begin{array}{cc}
                A_{11} & A_{12} \\
                A_{12}^* & A_{22}
              \end{array}
\right)$ be a Hermitian $(n+m)\times (n+m)$ matrix. Therefore, the
following statements are equivalent:
\begin{itemize}
 \item[i)] $A>0$.
 \item[ii)] $A_{11}>0$ and  $A_{12}^*A_{11}^{-1}A_{12}<A_{22}$.
\item[iii)] $A_{22}>0$ and  $A_{12}A_{22}^{-1}A_{12}^*<A_{11}$.
\end{itemize}
\end{lem}

%
%


\vskip2mm
 \noindent
\begin{minipage}{240pt}
 A.~E.~Choque-Rivero\\
 Instituto de F\'isica y Matem\'aticas\\
 Universidad Michoacana de San Nicol\'as de Hidalgo\\
 Ciudad Universitaria\\
 Morelia, Mich.\\
 C.P.~58048\\
 M\'exico\\
 \texttt{abdon@ifm.umich.mx}
\end{minipage}

\end{document}